\documentclass[11pt,a4paper]{article}

\usepackage{graphicx}        

\usepackage[a4paper,hmargin={2.7cm,2.7cm},vmargin={3.3cm,3.3cm}]{geometry}
\usepackage{amsmath, amssymb, amsfonts, amsbsy, latexsym, color,dsfont}
\usepackage{amscd,amsxtra}
\usepackage{amsthm}
\usepackage[normalem]{ulem}

\usepackage{verbatim}

\usepackage[T1]{fontenc}
\usepackage[latin1,utf8]{inputenc}
\usepackage[english]{babel}
\usepackage{cite}
\usepackage{paralist}
\usepackage{url}
\usepackage{twoopt}

\usepackage[bf,compact,pagestyles,small]{titlesec}
\titlespacing*{\section}{0pt}{14pt}{4pt}
\titlespacing*{\subsection}{0pt}{8pt}{3pt}

\usepackage{etoolbox}
\makeatletter
\patchcmd{\ttlh@hang}{\parindent\z@}{\parindent\z@\leavevmode}{}{}
\patchcmd{\ttlh@hang}{\noindent}{}{}{}
\makeatother

\usepackage{mathdots}

\usepackage{fancyhdr,lastpage}
\def\maketimestamp{\count255=\time
\divide\count255 by 60\relax
\edef\thetime{\the\count255:}%
\multiply\count255 by-60\relax
\advance\count255 by\time
\edef\thetime{\thetime\ifnum\count255<10 0\fi\the\count255}
\edef\thedate{\number\day-\ifcase\month\or Jan\or Feb\or Mar\or
             Apr\or May\or Jun\or Jul\or Aug\or Sep\or Oct\or
             Nov\or Dec\fi-\number\year}
\def\timstamp{\hbox to\hsize{\tt\hfil\thedate\hfil\thetime\hfil}}}
\maketimestamp
\fancypagestyle{plain}{%
\fancyhf{} 
\fancyfoot[C]{\small \thepage{} of \pageref{LastPage}}}

\numberwithin{equation}{section}  
\allowdisplaybreaks[4]

\newtheorem{theorem}{Theorem}[section]

\newtheorem{lemma}[theorem]{Lemma}
\newtheorem{proposition}[theorem]{Proposition}
\newtheorem{corollary}[theorem]{Corollary}
\theoremstyle{definition}
\newtheorem{definition}[theorem]{Definition} 
\newtheorem{example}{Example}[section]
\AtEndEnvironment{example}{\null\hfill\ensuremath{\blacksquare}}
\theoremstyle{remark}
\newtheorem{remark}[example]{Remark} 

\DeclareMathOperator{\supp}{supp} %
\DeclareMathOperator*{\esssup}{ess\,sup} %
\DeclareMathOperator*{\essinf}{ess\,inf} %

\DeclareMathOperator{\exponential}{e}

\newcommandtwoopt{\gaborG}[3][\Lambda][\Gamma]{\mathcal{G}(#3,#1,#2)} 

\newcommand{\myexp}[1]{\exponential^{#1}}

\newcommand{\eps}{\ensuremath{\varepsilon}}

\newcommand*{\numbersys}[1]{\ensuremath{\mathbb{#1}}}

\newcommand*{\R}{\numbersys{R}}

\newcommand*{\Q}{\numbersys{Q}}

\newcommand*{\Z}{\numbersys{Z}}

\newcommand*{\N}{\numbersys{N}}
\newcommand*{\T}{\numbersys{T}}

\newcommand{\itvoc}[2]{\ensuremath{\left({#1},{#2}\right]}} 
\newcommand{\itvcc}[2]{\ensuremath{\left[{#1},{#2}\right]}} %
\newcommand{\itvco}[2]{\ensuremath{\left[{#1},{#2}\right)}} %
\newcommand{\abs}[1]{\ensuremath{\left\lvert#1\right\rvert}}
\newcommand{\abssmall}[1]{\ensuremath{\lvert#1\rvert}}
\newcommand{\absbig}[1]{\ensuremath{\bigl\lvert#1\bigr\rvert}}
\newcommand{\absBig}[1]{\ensuremath{\Bigl\lvert#1\Bigr\rvert}}
\newcommand{\norm}[2][]{\ensuremath{\left\lVert#2\right\rVert_{#1}}}

\newcommand{\innerprod}[3][]{\ensuremath{\left\langle #2,#3\right\rangle_{\! #1}}}

\newcommand{\set}[1]{\ensuremath{\left\lbrace{#1}\right\rbrace}}

\newcommand{\setprop}[2]{\ensuremath{\left\lbrace{#1} : {#2}\right\rbrace}}

\newcommand{\setpropbig}[2]{\ensuremath{\bigl\lbrace{#1} :
      {#2}\bigr\rbrace}}
\newcommand{\setpropsmall}[2]{\ensuremath{\lbrace{#1} : {#2}\rbrace}}

\newcommand\cD{\mathcal{D}} 

\usepackage{xspace}

\makeatletter
\def\moverlay{\mathpalette\mov@rlay}
\def\mov@rlay#1#2{\leavevmode\vtop{%
   \baselineskip\z@skip \lineskiplimit-\maxdimen
   \ialign{\hfil$\m@th#1##$\hfil\cr#2\crcr}}}
\newcommand{\charfusion}[3][\mathord]{
    #1{\ifx#1\mathop\vphantom{#2}\fi
        \mathpalette\mov@rlay{#2\cr#3}
      }
    \ifx#1\mathop\expandafter\displaylimits\fi}
\makeatother

\newcommand{\bigcupdot}{\charfusion[\mathop]{\bigcup}{\cdot}}

\DeclareMathOperator*{\covol}{d}
\DeclareMathOperator*{\AP}{AP}
\DeclareMathOperator*{\aut}{Aut}
\DeclareMathOperator*{\loc}{loc}

\newlength{\dhatheight}

\newcommand{\ghat}{\widehat{G}}

\newcommand*\oline[1]{%
  \vbox{%
    \hrule height 0.5pt
    \kern0.25ex
    \hbox{%
      \kern-0.1em
      \ifmmode#1\else\ensuremath{#1}\fi
      \kern-0.1em
    }
  }
}

{\nopagebreak\par\noindent\hbox to \hsize{\hrulefill}\par\noindent}

\usepackage{hyperref}
\hypersetup{
 pdfview={FitH},
 pdfstartview={FitH},
 pdfauthor = {Jakob Lemvig, Jordy Timo van Velthoven},
 pdftitle = {Criteria for generalized translation-invariant frames},
 pdfsubject = {harmonic analysis},
 pdfkeywords = {almost periodic functions, frame, frame bound, Gabor system, generalized
   translation-invariant system, necessary condition, sufficient
   condition, wavelet system},
 pdfcreator = {LaTeX with hyperref package},
 pdfproducer = PDFlatex}

\makeatletter
\def\blfootnote{\xdef\@thefnmark{}\@footnotetext} 
\def\subjclass{\xdef\@thefnmark{}\@footnotetext}
\long\def\symbolfootnote[#1]#2{\begingroup%
\def\thefootnote{\fnsymbol{footnote}}\footnote[#1]{#2}\endgroup} 
\if@titlepage
  \renewenvironment{abstract}{%
      \titlepage
      \null\vfil
      \@beginparpenalty\@lowpenalty
      \begin{center}%
        \bfseries \abstractname
        \@endparpenalty\@M
      \end{center}}%
     {\par\vfil\null\endtitlepage}
\else
  \renewenvironment{abstract}{%
      \if@twocolumn
        \section*{\abstractname}%
      \else
        \small
        \list{}{%
          \settowidth{\labelwidth}{\textbf{\abstractname:}}
          \setlength{\leftmargin}{50pt}
          \setlength{\rightmargin}{50pt}
          \setlength{\itemindent}{\labelwidth}
          \addtolength{\itemindent}{\labelsep}
        }
        \item[\textbf{\abstractname:}]

      \fi}
      {\if@twocolumn\else\endlist\fi}
\fi
\makeatother

\usepackage{color}

\begin{document}

\title{Criteria for generalized translation-invariant frames}

\date{}

\thispagestyle{plain}

\author{
Jakob Lemvig\footnote{Corresponding author. Technical University of Denmark, Department of Applied Mathematics and Computer Science, Matematiktorvet 303B, 2800 Kgs.\ Lyngby, Denmark, \mbox{E-mail: \protect\url{jakle@dtu.dk}}}
, Jordy Timo van Velthoven\footnote{University of Vienna, Faculty of Mathematics, Oskar-Morgenstern-Platz 1, A-1090 Vienna, Austria,  \mbox{E-mail: \protect\url{jordy-timo.van-velthoven@univie.ac.at}}}
}
\maketitle 

 \blfootnote{2010 {\it Mathematics Subject Classification.} Primary
   42C15, 42C40, Secondary: 43A32, 43A60, 43A70, 94A12.} \blfootnote{{\it Key  words and phrases.}
   almost periodic functions, frame, frame bound, Gabor system, generalized
   translation-invariant system, necessary condition, sufficient
   condition, wavelet system}

\begin{abstract}
  This paper provides new sufficient and necessary conditions for the
  frame property of generalized trans\-la\-tion-in\-vari\-ant
  systems. The conditions are formulated in the Fourier domain and
  consists of estimates involving the upper and lower frame
  bound. Contrary to known conditions of a similar nature, the
  estimates take the phase of the generating functions in
  consideration and not only their modulus. The possibility of phase
  cancellations makes these estimates optimal for tight frames. The
  results on generalized trans\-la\-tion-in\-vari\-ant systems will be
  proved in the setting of locally compact abelian groups, but even
  for Euclidean space and the special case of (composite) wavelet
  systems the results are new.
\end{abstract}

\section{Introduction}
Deriving sufficient and necessary conditions for the frame property of
structured function systems has a long history in time-frequency and
time-scale analysis. In this paper we study a class of structured
function systems known as generalized trans\-la\-tion-in\-vari\-ant
systems. These function systems offer a common framework for discrete
and continuous structured function systems such as Gabor systems, composite wavelet systems 
and wave packet systems, but are also of
independent interest.

\subsection{Overview and contributions}
\label{sec:overview}
The paper aims to derive necessary and sufficient conditions for the
frame and Bessel property of generalized trans\-la\-tion-in\-vari\-ant
systems that are based on properties of the generating functions in
the Fourier domain. The first results similar in nature go back to the
very beginning of modern frame theory and the influential papers by
Daubechies~\cite{Dau90} and Daubechies, Grossmann and
Meyer~\cite{MR836025}. In~\cite{Dau90}, Daubechies provides general
conditions on the generators and parameters of Gabor and wavelet
systems to form a Bessel system or a frame for $L^2 (\mathbb{R})$.
These fundamental results attracted the attention of several groups of
researchers~\cite{heil1989continuous, MR1353539, MR1784021,
  ron1995gramian, MR1460623, chui1993a, chui1993inequalities,
  jing1999on} and lead to improvements and generalizations over the
subsequent decades, whose precise nature is discussed in
Section~\ref{sec:related-work}.

For the sake of clarity, in the remainder of this introduction, we
focus on a subclass of \emph{discrete} function systems, called generalized
shift-invariant systems. 
In the setting of a locally compact abelian group $G$, written additively, a generalized
shift-invariant system in $L^2(G)$ is a countable union of the form  
 \[
 \bigcup_{j \in J} \bigg\{g_j (\cdot - \gamma) \; : \; \gamma \in
 \Gamma_j \bigg\}
 \]
 for a collection of discrete, co-compact subgroups
 $\{\Gamma_j \}_{j \in J}$ in $G$ and a family of functions
 $ \{g_j\}_{j \in J}$ in $L^2 (G)$. A generalized shift-invariant
 system
 $\cup_{j \in J} \{g_j (\cdot - \gamma) \}_{\gamma \in \Gamma_j}$ in
 $L^2 (G)$ is called a \emph{frame} for $L^2 (G)$ whenever there exist
 constants $A, B > 0$, called \emph{frame bounds}, such that
\begin{align} \label{eq:frame_ineq_intro}
A \| f \|^2_2 \leq \sum_{j \in J} \sum_{\gamma \in \Gamma_j} |\langle f, g_j (\cdot - \gamma) \rangle|^2 \leq B \| f \|_2^2
\end{align}
for all $f \in L^2 (G)$. A family $\cup_{j \in J} \{g_j
(\cdot - \gamma) \}_{\gamma \in \Gamma_j}$ satisfying the
upper frame bound is called a \emph{Bessel sequence} and a frame for
which the frame bounds can be chosen equal is called
\emph{tight}. 

Frames in $L^2(G)$ are of interest in applications in, e.g., signal
analysis and functional analysis, as they guarantee unconditionally
$L^2$-convergent and stable expansions of functions in
$L^2(G)$. Indeed, given a frame
$\cup_{j \in J} \{g_j (\cdot - \gamma) \}_{\gamma \in \Gamma_j}$ in
$L^2 (G)$, there exists a system
$\cup_{j \in J} \{\tilde{g}_{j, \gamma} \}_{\gamma \in \Gamma_j}$ such
that every function $f \in L^2 (G)$ possesses an expansion of the form
\begin{equation}
f =  \sum_{j \in J} \sum_{\gamma \in \Gamma_j} \langle f,
\tilde{g}_{j, \gamma} \rangle  g_j (\cdot - \gamma)= \sum_{j \in J}
\sum_{\gamma \in \Gamma_j} \langle f,   g_j (\cdot - \gamma) \rangle
\tilde{g}_{j, \gamma}\label{eq:frame-expansion}
\end{equation}
with unconditional norm convergence.  

Verifying the frame inequalities \eqref{eq:frame_ineq_intro} directly
is often an impossible task. However, for many special cases simple
sufficient and necessary conditions for the frame property are known.
The new criteria presented in this paper will be derived under two
unconditionally convergence properties (called $p$-UCP for
$p=1,\infty$) that will be introduced in Definition \ref{def:ucp}.
The $p$-UCP is a mild convergence property that guarantees almost
periodicity of an auxiliary function $w_f$, essential to our analysis,
that will be introduced in Section~\ref{sec:uncond-conv-prop}. The
$\infty$-UCP will simply be assumed throughout the remainder of this
subsection; here, for simplicity, we also assume that the Lebesgue
Differentiation Theorem holds on $\ghat$.  The frame criteria are
phrased as estimates involving the functions
$t_{\alpha} : \widehat{G} \to \mathbb{C}$,
$\alpha \in \bigcup_{j \in J} \Gamma_j^\perp$, defined, whenever
convergent, as
\[
t_{\alpha} (\omega) = \sum_{j \in J \; : \;\alpha \in
    \Gamma_j^\perp} \frac{1}{\covol(\Gamma_j)} \hat{g}_j (\omega)
\overline{\hat{g}_j (\omega + \alpha)}, \qquad \text{a.e.\ } \omega \in \ghat,
\]
where $\Gamma_j^\perp$ is the dual lattice of $\Gamma_j$, and
$\covol(\Gamma_j)$ denotes the covolume of $\Gamma_j$.

The sufficient condition presented in
Theorem~\ref{thm:GTIsufficient_CC} states that a system
$\cup_{j \in J} \{g_j (\cdot - \gamma) \}_{\gamma \in \Gamma_j}$ forms
a frame for $L^2(G)$ with bounds $A_1$ and $B_1$ if
\begin{align} \label{eq:CC_intro}
B_1 := \esssup_{\omega \in \ghat} \sum_{\alpha \in \bigcup_{j \in J} \Gamma_j^\perp} |t_{\alpha} (\omega) | < \infty
\end{align}
and
\begin{align} \label{eq:CC_A_intro}
A_1 := \essinf_{\omega \in \ghat} \bigg( t_0 (\omega) - \sum_{\alpha \in \bigcup_{j \in J} \Gamma_j^\perp \setminus \{ 0 \}} |t_{\alpha} (\omega) |  \bigg) > 0.
\end{align}
The subscript $p$ is used in the constants $A_p$ and $B_p$ to indicate
the relation with the $\ell_p$ norm of the sequence
$\{t_{\alpha} (\omega)\}_{\alpha}$ for a.e. $\omega \in \ghat$.
Contrary to previously known sufficient conditions for generalized
shift-invariant systems, the estimates \eqref{eq:CC_intro} and
\eqref{eq:CC_A_intro} take the phase of the generating functions in
consideration and not only their modulus. To be more precise, the
previously known sufficient conditions are not based on the functions
$t_\alpha : \ghat \to \mathbb{C}$, but are based on the non-negative
functions
\[
\ghat \ni \omega \mapsto \sum_{j \in J \; : \; \alpha \in \Gamma_j^\perp} \frac{1}{\covol(\Gamma_j)} \abs{\hat{g}_j (\omega) \hat{g}_j (\omega + \alpha)} \in [0,\infty[.
\]
By considering the phase of the generating functions, the estimates
\eqref{eq:CC_intro} and \eqref{eq:CC_A_intro} allow, for each
$\alpha \in \bigcup_{j \in J} \Gamma_j^\perp$, for phase cancellations
in the sum over $\setpropbig{j \in J}{ \alpha \in \Gamma_j^\perp }$
and therefore often lead to improvements of the known estimates, see,
e.g., Example~\ref{ex:CC_ACC} for an orthonormal basis demonstrating
this.  In fact, the estimates~\eqref{eq:CC_intro} and
\eqref{eq:CC_A_intro} are optimal for tight frames in the sense that
they recover precisely the frame bound.
   
The obtained necessary condition (Theorem~\ref{thm:GTInec_little-l2})
asserts that if
$\cup_{j \in J} \{g_j (\cdot - \gamma) \}_{\gamma \in \Gamma_j}$ is a
Bessel sequence in $L^2 (G)$ with upper frame bound $B$, then
\begin{align}\label{eq:B_1-lower_intro}
 B &\ge B_2:= \esssup_{\omega \in \ghat} \left(\sum_{\alpha \in
  \bigcup_{j \in J} \Gamma_j^\perp}
       \abs{t_{\alpha}(\omega)}^2\right)^{1/2}. 
 \end{align}
 Combining this with the known fact that for a frame
 $\cup_{j \in J} \{g_j (\cdot - \gamma) \}_{\gamma \in \Gamma_j}$ for
 $L^2 (G)$ with lower bound $A >0$ necessarily
 $A_{\infty} := \essinf_{\omega \in \ghat} t_0 (\omega) \geq A$, it
 follows that $A \le A_{\infty} \le B_2 \le B$ is necessary for a
 frame
 $\cup_{j \in J} \{g_j (\cdot - \gamma) \}_{\gamma \in \Gamma_j}$ with
 bounds $A$ and $B$.

 For the applicability of the frame
 expansions~\eqref{eq:frame-expansion}, it is not only essential to
 verify the frame inequalities~\eqref{eq:frame_ineq_intro}, but also
 to provide good estimates of the frame bounds. The obtained necessary
 and sufficient conditions yield together \emph{frame bound estimates}
 for generalized shift-invariant systems.  Indeed, for a Bessel system
 $\cup_{j \in J} \{g_j (\cdot - \gamma) \}_{\gamma \in \Gamma_j}$ with
 optimal upper bound $B > 0$, the bound $B$ can be estimated by the
 snug bounds
\begin{equation}
 B_2 \equiv  \esssup_{\omega \in \widehat{G}} \norm[\ell^2]{\{t_\alpha(\omega)\}_\alpha} \le B \le
\esssup_{\omega \in \widehat{G}} \norm[\ell^1]{\{t_\alpha(\omega)\}_\alpha} \equiv
B_1.
\label{eq:Bessel-snug-estimates}
\end{equation}
If, furthermore, $\cup_{j \in J} \{g_j (\cdot - \gamma)
\}_{\gamma \in \Gamma_j}$ is a frame in $L^2 (G)$ with optimal lower
bound $A > 0$, then 
\begin{equation}
 A_1 \le A \le A_\infty.
\label{eq:lower-snug-estimates}
\end{equation}

The presented results hold not only for discrete frames and systems as
discussed above, but also for their continuous and semi-continuous
counterpart. To summarize, the main contributions of the paper are new
necessary and sufficient conditions for the frame and Bessel property
of generalized trans\-la\-tion-in\-vari\-ant systems that are (i)
derived under minimal assumptions, (ii) optimal for tight frames,
(iii) verifiable and computable, and that provide (iv) snug frame
bound estimates which collapse to equality for tight frames.

\subsection{Related work}
\label{sec:related-work}
For shift-invariant systems, i.e., generalized shift-invariant systems
with a single, fixed translation lattice $\Gamma$, the Bessel and
frame properties can be \emph{characterized} in terms of (bi-infinite)
matrix-valued functions, known as \emph{dual Gramian matrices}, as
introduced by Ron and Shen \cite{MR1350650}, see also Janssen
\cite{MR1601115}.  Consequently, the aforementioned necessary and
sufficient conditions in \eqref{eq:CC_intro}, \eqref{eq:CC_A_intro}
and \eqref{eq:B_1-lower_intro} can be derived\footnote{For example,
  the necessary condition in \eqref{eq:B_1-lower_intro} for
  shift-invariant systems follows from the norm estimate
  $\norm{M}\ge \left(\sum_{j \in \Gamma^\perp}
    \abs{m_{0,j}}^2\right)^{1/2}$
  by noticing that the $0$th column the dual Gramian matrix at
  $\omega\in \ghat$ is
  $\{t_\alpha(\omega)\}_{\alpha \in \Gamma^\perp}$.} from simple norm
estimates of bi-infinite Hermitian matrices
$M=(m_{i,j})_{i,j\in \Gamma^{\perp}}$ on $\ell^2(\Gamma^\perp)$, see
\cite[Section 1.6]{MR1350650}. In particular, the estimates
\eqref{eq:Bessel-snug-estimates} and \eqref{eq:lower-snug-estimates}
are known for separable Gabor systems \cite{MR1350700, MR1460623};
these estimates are the best-known improvement of Daubechies' Gabor
frame bound estimates~\cite{Dau90}.  Furthermore, the dual Gramian
characterization has, in a fiberization formulation~\cite{MR1795633},
been extended to the setting of locally compact abelian groups
\cite{BowRos2014,MR2578463, MR344891}. Hence, for shift-invariant
systems, or more generally, trans\-la\-tion-in\-vari\-ant systems on
such groups, the conditions \eqref{eq:CC_intro}, \eqref{eq:CC_A_intro}
and \eqref{eq:B_1-lower_intro} follow from these characterizations and
should not be considered new.

For function systems that are not shift-invariant, the fiberization
characterization breaks down. In spite of this, Ron and
Shen~\cite{MR2132766} obtained dual Gramian-type\footnote{For
  generalized shift-invariant systems the direct link using \emph{one}
  dual Gramian matrix for each fiber has to be replaced by a less
  direct link of infinite families of finite matrices for each fiber.}
characterizations for special types of generalized shift-invariant
systems in $L^2(\R^d)$. For example, for generalized shift-invariant
systems satisfying the finite intersection (FI) condition (i.e., the
intersection of any finite subfamily of the lattices
$\{\Gamma_j\}_{j \in J}$ is a full-rank lattice), the Bessel property
can be characterized by the norm of the dual Gramian matrices since
the FI condition essentially reduces the analysis to standard dual
Gramian analysis. On the other hand, many generalized shift-invariant
systems violate the FI condition, e.g., systems with both rational and
non-rational lattices. For lower frame bound characterizations by dual
Gramian analysis additional restrictions on the lattices and
generators are needed, most notably the so-called small tail condition. To
handle wavelet systems associated with expansive, but not necessarily
integer, matrix dilations, other assumptions on the family of lattices
and generators are made such as the notions of temperateness and
roundedness in generalized shift-invariant systems,
cf. \cite{MR2132766} for definitions. However, we stress that none of
the used assumptions in \cite{MR2132766} are weak enough to allow for
dual Gramian characterization of wavelet frames associated with
arbitrary real, expansive dilations.

An alternative route for deriving necessary and sufficient conditions
for wavelet frames with integer, expansive dilations goes through
quasi-affine systems
\cite{ron1997affine,ron1997affine2,chui1998affine}. This link is known
to generalize to rational, expansive dilations \cite{MR2746669},
although one has to consider a family of quasi-affine systems to
capture the frame property of the given wavelet system.  Since
quasi-affine systems are shift-invariant, sufficient and necessary
conditions for rational wavelet systems are readily available. We
stress that such estimates differ slightly from the ones presented in
this paper.  The estimates presented in \cite{ron1995gramian} for
wavelet systems with integer, expansive dilations utilize the
quasi-affine route, but they ignore the phase of the wavelet generator
and are therefore not optimal for tight frames.

The approach we follow relies on a connection between the frame
properties of generalized shift-invariant systems and an associated
almost periodic auxiliary function \cite{MR1940326, MR1916862,
  MR1310658, MR1353539}. Our methods are closely related to the work
of Hern\'andez, Labate and Weiss~\cite{MR1916862}, but while
\cite{MR1916862} is concerned with tight frame characterization using
uniqueness of the coefficients of almost periodic Fourier series, we
focus on non-tight frames by bounding the Fourier series.  The
connection to Fourier analysis is valid under the 1-UCP, which is weak
enough to provide sufficient and necessary conditions for wavelet
frames in $L^2(\R^d)$ associated with any real, expansive dilation
matrix and any translation lattice. The $1$-UCP is even weak enough to
handle every choice of real, invertible dilation (not necessarily
expansive) almost surely with respect to the Haar measure on
$\mathrm{GL}_d (\mathbb{R})$, see Section~\ref{sec:wavelet-systems}.

To wrap up the discussion, no necessary or sufficient conditions,
optimal for tight frames, are currently known for wavelet systems
associated with expansive, real dilations. In fact, the lack of
optimal frame bound estimates for such systems lead
Christensen~\cite{MR1824891} to ask whether sufficient conditions as
in \eqref{eq:CC_intro} and \eqref{eq:CC_A_intro} can also be obtained
for wavelet systems with non-integer dilations. The sufficient
conditions obtained in this paper answer this question in the
affirmative.

\subsection{Outline}
\label{sec:outline}

The paper is organized as follows. In
Section~\ref{sec:gener-transl-invar} we introduce generalized
trans\-la\-tion-in\-vari\-ant systems and the $1$-UCP condition in the
setting of locally compact abelian groups. The main results on
generalized trans\-la\-tion-in\-vari\-ant systems are presented in
Section~\ref{sec:suff_nec}. Necessary and sufficient conditions for
generalized trans\-la\-tion-in\-vari\-ant frames are contained in
Section~\ref{sec:necessary-conditions} and \ref{sec:main-result},
respectively.  In Section~\ref{sec:cc-cond-absol} we compare the
obtained frame bound estimates with known
estimates. Section~\ref{sec:examples} is devoted to applications and
examples. Gabor systems and wavelet systems are considered in
Section~\ref{sec:gabor-systems} and Section~\ref{sec:wavelet-systems},
respectively. Finally, we consider composite wavelet and cone-adapted
shearlet systems in Section~\ref{sec:shearlet-systems}, and we derive
new frame characterizations of the continuous $\ell$-th order
$\alpha$-shearlet transform in Section~\ref{sec:continuous-transform}.

\section{Generalized translation-invariant systems}
\label{sec:gener-transl-invar}
Throughout this paper, $G$ will denote a second countable locally
compact abelian group. The character group of $G$ is denoted by
$\widehat{G}$ and forms a second countable locally compact abelian
group itself. Both $G$ and $\widehat{G}$ are $\sigma$-compact and metrizable. 
The group operation in both $G$ and $\widehat{G}$ is
written additively as $+$ and the identity element is denoted by
$0$. The Haar measure on $G$ will be denoted by $\mu_G$. It is assumed
that the Haar measure on $G$ is given and that the Haar measure on
$\widehat{G}$ is the Plancherel measure. The subset
$\Gamma \subseteq G$ will denote a closed, co-compact subgroup of $G$,
i.e., the quotient space $G / \Gamma$ is compact. In this case, the
annihilator $\Gamma^{\perp}$ of $\Gamma$ is the countable, discrete
subgroup
$\Gamma^{\perp} := \{ \omega \in \widehat{G} \; | \; \omega (x) = 0,
\; \forall x \in \Gamma \}$.
It is assumed that the Haar measure on $\Gamma$ is given and that the
Haar measure on $G / \Gamma$ is the unique quotient measure provided
by Weil's integral formula. Using this quotient measure
$\mu_{G/ \Gamma}$ on $G / \Gamma$, the \emph{covolume} or \emph{size}
of the subgroup $\Gamma \subseteq G$ is defined as
$\covol(\Gamma) := \mu_{G / \Gamma} (G / \Gamma)$.

\subsection{Generalized translation-invariant frames}
\label{sec:gener-transl-invar-1}
The function systems defined next form the central object of this paper. Here, the translate of a function $f \in L^2 (G)$ by $y \in G$ is denoted as $T_y f := f(\cdot - y)$.

\begin{definition}
  Let $J$ be a countable index set. For each $j \in J$, let
  $\Gamma_j \subseteq G$ be a closed, co-compact subgroup, and let
  $P_j$ be an arbitrary (countable or uncountable) index set. For a
  given family of functions
  $\cup_{j \in J} \{g_{j, p} \}_{p \in P_j} \subset L^2 (G)$, the
  collection of translates
\[ \bigcup_{j \in J} \{T_{\gamma} g_{j, p}  \}_{\gamma \in \Gamma_j, p \in P_j} \]
is called a \emph{generalized trans\-la\-tion-in\-vari\-ant (GTI) system} in $L^2 (G)$. 
\end{definition}

We stress that the translation subgroups $\Gamma_j \subset G$,
$j \in J,$ are not assumed to be discrete, but merely closed and
co-compact.  In $G = \mathbb{R}^d$ (and $\widehat{G} = \mathbb{R}^d$),
each closed, co-compact subgroup $\Gamma_j$ is of the form
$C (\mathbb{Z}^k \times \mathbb{R}^{d-k})$ for some $0 \leq k \leq d$
and $C \in \mathrm{GL}_d (\mathbb{R})$.  The (discrete) annihilator
$\Gamma_j^\perp$ is then given by
$(C^T)^{-1}(\Z^k \times \{0\}^{d-k})$.

Following~\cite{JakobsenReproducing2014}, it is assumed that the
generating functions of a generalized trans\-la\-tion-in\-vari\-ant
systems satisfy the following three \emph{standing hypotheses}.  For
each $j \in J$:
\begin{enumerate}[(I)]
\item The triple $(P_j, \Sigma_{P_j}, \mu_{P_j})$ forms a
  $\sigma$-finite measure space; \label{item:hyp1}
\item The mapping $p \mapsto g_{j,p}$ from $(P_j, \Sigma_{P_j})$ into
  $(L^2 (G), \mathcal{B}_{L^2 (G)})$ is $\Sigma_{P_j}$-measurable,
  where $\mathcal{B}_{L^2 (G)}$ denotes the Borel $\sigma$-algebra on
  $L^2(G)$; \label{item:hyp2}
\item The mapping $(p, x) \mapsto g_{j,p} (x)$ from
  $(P_j \times G, \Sigma_{P_j} \otimes \mathcal{B}_G)$ into
  $ (\mathbb{C}, \mathcal{B}_{\mathbb{C}})$ is
  $(\Sigma_{P_j} \otimes \mathcal{B}_G)$-measurable, where
  $\mathcal{B}_G$ denotes the Borel $\sigma$-algebra on
  $G$; \label{item:hyp3}
\end{enumerate}

For most applications in this paper, it suffices to take $\{P_j\}_{j \in J}$ 
to be countable index sets equipped with the counting measure, cf. 
Section \ref{sec:examples}. In this case, the three standing hypotheses \eqref{item:hyp1}--\eqref{item:hyp3} are automatically satisfied.

A generalized trans\-la\-tion-in\-vari\-ant system
$\cup_{j \in J} \{T_{\gamma} g_{j,p} \}_{\gamma \in \Gamma_j, p \in
  P_j}$
is called a \emph{generalized trans\-la\-tion-in\-vari\-ant frame} for
$L^2 (G)$, with respect to
$\{L^2 (P_j \times \Gamma_j) \; | \; j \in J \}$, whenever there exist
two constants $A, B > 0$, called the \emph{frame bounds}, such that
\begin{align} \label{eq:GTIframe}
A\| f \|^2 \leq \sum_{j \in J} \int_{P_j} \int_{\Gamma_j} | \langle f, T_{\gamma} g_{j,p} \rangle |^2 \; d\mu_{\Gamma_j} (\gamma) d\mu_{P_j} (p) \leq B \|f\|^2
\end{align}
for all $f \in L^2 (G)$. A generalized trans\-la\-tion-in\-vari\-ant
system
$\cup_{j \in J} \{T_{\gamma} g_{j,p} \}_{\gamma \in \Gamma_j, p \in
  P_j}$
satisfying the upper frame bound is called a \emph{Bessel system} or a
\emph{Bessel family} in $L^2 (G)$. For such a Bessel system, the frame
operator $S:L^2(G)\to L^2(G)$, associated with
$\cup_{j \in J} \{T_{\gamma} g_{j,p} \}_{\gamma \in \Gamma_j, p \in
  P_j}$,
is defined weakly by equating $\innerprod{Sf}{f}$ with the central
term of~\eqref{eq:GTIframe}.

In order to check whether a generalized trans\-la\-tion-in\-vari\-ant
system forms a Bessel system or a frame for $L^2 (G)$, it suffices to
check the frame condition on a dense subspace of $L^2 (G)$. Let
$\mathcal{E}$ denote the set of all closed Borel sets
$E \subseteq \widehat{G}$ satisfying $\mu_{\widehat{G}} (E) = 0$. For
a fixed $E \in \mathcal{E}$, define the dense subspace
$\mathcal{D}_E (G)$ of $L^2 (G)$ as
\[
 \mathcal{D}_E (G) = \left\{ f \in L^2 (G) \;| \; \hat{f} \in L^\infty(\ghat) \text{ and }
    \exists K \subset \ghat\setminus E \text{ compact with } \hat{f}
    \mathds{1}_K = \hat{f} \text{ a.e.}\right\}  ,
\]
where $\hat{f}$ denotes the Fourier transform of $f \in L^2 (G)$ and
$\mathds{1}_K$ the characteristic function on $K$.  A function
$f \in \mathcal{D}_E (G)$ has the property that $\hat{f}$ is zero
almost everywhere on $E+B(0,\delta)$ for some $\delta>0$, where
$B(0,\delta)$ denotes the open ball of radius $\delta$, since
$\supp \hat{f}$ is compact, $E$ is closed, and
$\supp{\hat{f}} \cap E = \emptyset$.

We consider the set $E \in \mathcal{E}$, called the \emph{blind spot},
as fixed, but arbitrary.  The actual choice of $E$ depends on the
considered application. For wavelet systems, the blind spot $E$ is
usually chosen to be the complement $\mathcal{O}^c$ of the dual orbit
$\mathcal{O} \subseteq \widehat{G}$ of the dilation group
\cite{fuehr2010generalized}, whereas for Gabor systems it usually
suffices to take $E = \emptyset$.

\subsection{Unconditional convergence property and Fourier analysis}
\label{sec:uncond-conv-prop}
This section is devoted to the $1$-UCP condition and to generalized
Fourier series of almost periodic functions essential to our study of
generalized translation-invariant frames.  Recall that an almost
periodic function is the uniform limit of trigonometric polynomials.
The space of almost periodic functions is denoted by $\AP(G)$.

Given a generalized trans\-la\-tion-in\-vari\-ant system
$\cup_{j \in J} \{T_{\gamma} g_{j,p} \}_{\gamma \in \Gamma_j, p \in
  P_j}$
and a function $f \in \mathcal{D}_E (G)$, define, for $j \in J$, the
map $w_{f,j} : G \to [0,\infty]$ by
\begin{align} \label{eq:wfj}
 w_{f,j} (x) = \int_{P_j} \int_{\Gamma_j} | \langle T_x f, T_{\gamma} g_{j,p} \rangle |^2 \;  d\mu_{\Gamma_j} (\gamma) d\mu_{P_j}(p).
\end{align}
By the standing hypotheses \eqref{item:hyp1}--\eqref{item:hyp3}, the
integrals in \eqref{eq:wfj} are well-defined. Throughout this section, we will further assume that 
\begin{equation}
\int_{P_j} | \hat{g}_{j,p} (\cdot) |^2 \; d\mu_{P_j} (p) \in
L^1_{\loc} (\widehat{G}).  \label{eq:int-over-Pj-loc-int}
\end{equation}
It can be shown, using arguments
from~\cite{MR2283810,JakobsenReproducing2014}, that the integrability
condition~\eqref{eq:int-over-Pj-loc-int} ensures that each $w_{f,j}$
is a trigonometric polynomial, in particular, that $w_{f,j}$ are
$\Gamma_j$-periodic, continuous, and bounded.

Next, we define $w_f : G \to [0,\infty]$ as the following sum of
non-negative, trigonometric polynomials
\[
w_f (x) = \sum_{j \in J} w_{f,j} (x) = \sum_{j \in J}  \int_{P_j} \int_{\Gamma_j} | \langle T_x f, T_{\gamma} g_{j,p} \rangle |^2 \; d\mu_{\Gamma_j} (\gamma) d\mu_{P_j}(p).
\]
The function $w_f : G \to [0,\infty]$ is well-defined, but it might
attain the value of positive infinity without any further assumptions
on the generalized trans\-la\-tion-in\-vari\-ant system.

Under a suitable regularity condition on the system
$\cup_{j \in J} \{T_{\gamma} g_{j,p} \}_{\gamma \in \Gamma_j, p \in
  P_j}$,
the function $w_f$ becomes almost periodic.  The connection between
the almost periodicity of $w_f$ and generalized
trans\-la\-tion-in\-vari\-ant frames was first used by
Laugesen~\cite{MR1866351, MR1940326} for wavelet systems and extended
to arbitrary generalized shift-invariant (GSI) systems in $L^2(\R^d)$
by Hern\'andez, Labate and Weiss~\cite{MR1916862}.

The following regularity condition, sufficiently weak for our
purposes, was introduced in \cite{bandwidth}.
\begin{definition} \label{def:ucp}
Let $\cup_{j \in J} \{T_{\gamma} g_{j,p} \}_{\gamma \in \Gamma_j, p \in P_j}$ be a generalized trans\-la\-tion-in\-vari\-ant system.
\begin{enumerate}[(i)]
\item The system
  $\cup_{j \in J} \{T_{\gamma} g_{j,p} \}_{\gamma \in \Gamma_j, p \in
    P_j}$
  is said to satisfy the \emph{1-unconditional convergence property}
  (1-UCP), with respect to $E \in \mathcal{E}$, whenever, for all
  $f \in \mathcal{D}_E (G)$, the function $w_f : G \to \mathbb{C}$ is
  almost periodic and the series
\begin{align} \label{eq:UCP}
w_f = \sum_{j \in J} w_{f,j}
\end{align}
converges unconditionally with respect to the mean
$M : \AP (G) \to [0,\infty)$,
\[
M (|f|) = \lim_{n \to \infty} \frac{1}{\mu_G (H_n)} \int_{H_n} |f(x)| \; d\mu_G (x),
\]
where $(H_n)_{n \in \mathbb{N}}$ is any increasing sequence of open, relatively compact subsets $H_n \subseteq G$ with $G = \bigcup_{n \in \mathbb{N}} H_n$ and such that
\[
\lim_{n \to \infty} \frac{\mu_G ((x + H_n) \cap (G \setminus H_n))}{\mu_G (H_n)} = 0
\]
for all $x \in G$. 
\item In case \eqref{eq:UCP} holds
   with uniform convergence, the system $\cup_{j \in J} \{T_{\gamma} g_{j,p} \}_{\gamma \in \Gamma_j, p \in P_j}$ is said to satisfy the \emph{$\infty$-UCP} with respect to $E \in \mathcal{E}$. 
   \end{enumerate}
\end{definition}

The $1$-UCP is the weakest known assumption under which Fourier
analysis of $w_f$ can be exploited to study the frame properties of
the associated system. Note that almost periodicity of $w_f$ is
assumed in $1$-UCP, while, in $\infty$-UCP, it follows from the
uniform convergence of~\eqref{eq:UCP}.

Both $1$-UCP and $\infty$-UCP are automatically satisfied whenever the
generalized trans\-la\-tion-in\-vari\-ant system
$\cup_{j \in J} \{T_{\gamma} g_{j,p} \}_{\gamma \in \Gamma_j, p \in
  P_j}$
satisfies the local integrability condition (LIC) or the weaker
$\alpha$-local integrability condition ($\alpha$-LIC) introduced
in~\cite{MR1916862} and \cite{JakobsenReproducing2014}, respectively.
A system
$\cup_{j \in J} \{T_{\gamma} g_{j,p} \}_{\gamma \in \Gamma_j, p \in
  P_j}$
is said to satisfy the \emph{$\alpha$-LIC}, with respect to $E$, if
\begin{equation*}
\sum_{j \in J} \frac{1}{\covol (\Gamma_j)} \int_{P_j} \sum_{\alpha
   \in \Gamma_j^{\perp}} \int_{\widehat{G}} \abs{ \hat{f} (\omega)
 \hat{f}(\omega + \alpha) \hat{g}_{j, p} (\omega + \alpha) \hat{g}_{j,
   p} (\omega)} \; d\mu_{\widehat{G}} (\omega) d\mu_{P_j} (p) <
 \infty 
\end{equation*}
for all $f \in \mathcal{D}_E (G)$.

The generalized Fourier series of $w_f$ is stated
in the next result, adapted from \cite{bandwidth}.

\begin{proposition} \label{prop:GTIgeneralisedfourierseries}
Suppose $\cup_{j \in J} \{T_{\gamma} g_{j,p} \}_{\gamma \in \Gamma_j, p \in P_j}$ satisfies the $1$-UCP with respect to $E \in \mathcal{E}$. Moreover, suppose that
\begin{align} \label{eq:Calderon_LI}
\sum_{j \in J} \frac{1}{\covol(\Gamma_j)} \int_{P_j} |\hat{g}_{j,p} (\cdot) |^2 \; d\mu_{P_j} (p) \in L^1_{\loc} (\widehat{G} \setminus E). 
\end{align}
Then, for all $f \in \mathcal{D}_E (G)$, the Fourier-Bohr transform of $w_f : G \to \mathbb{C}$ at $\alpha \in \widehat{G}$ is 
\[
\widehat{w_f} (\alpha) := M( w_f \cdot \overline{\alpha}) = \sum_{j \in J} \widehat{w_{f,j}} (\alpha)
\]
with absolute convergence. The generalized Fourier series of $w_f : G \to \mathbb{C}$ is given by
\begin{align} \label{eq:GTIgeneralisedfourierseries}
w_f  = \sum_{\alpha \in \bigcup_{j \in J} \Gamma_j^{\perp}} c_{\alpha} \alpha ,
\end{align}
where the Fourier coefficients are given by 
\begin{align} \label{eq:Fourier_coefficients} c_{\alpha} =
  \int_{\widehat{G}} \hat{f} (\omega) \overline{\hat{f}(\omega +
    \alpha)} \sum_{j \in J \; : \; \alpha \in \Gamma_j^{\perp}}
  \frac{1}{\covol(\Gamma_j)} \int_{P_j} \overline{\hat{g}_{j,p}
    (\omega)} \hat{g}_{j,p} (\omega + \alpha) \; d\mu_{P_j} (p)
  d\mu_{\widehat{G}} (\omega).
\end{align}
  Furthermore, if the
$\infty$-UCP holds, then $w_f$ agrees pointwise with its generalized Fourier
series \eqref{eq:GTIgeneralisedfourierseries}.  
\begin{proof}
  The result, except for the specific form of the Fourier
  coefficients, can be found in \cite{bandwidth}.  Indeed, by
  \cite[Proposition 3.10]{bandwidth}, the generalized Fourier series
  of $w_f : G \to \mathbb{C}$ is
  $w_f = \sum_{\alpha \in \bigcup_{j \in J} \Gamma_j^{\perp}}
  c_{\alpha} \alpha$ with coefficients
\[
c_{\alpha} = \sum_{j \in J \; : \; \alpha \in \Gamma_j^{\perp}}
\int_{\widehat{G}} \hat{f} (\omega) \overline{\hat{f}(\omega +
  \alpha)} \frac{1}{\covol(\Gamma_j)} \int_{P_j}
\overline{\hat{g}_{j,p} (\omega)} \hat{g}_{j,p} (\omega + \alpha) \;
d\mu_{P_j} (p) d\mu_{\widehat{G}} (\omega).
\]
For fixed $\alpha \in \bigcup_{j \in J} \Gamma_j^{\perp}$, the
assumption \eqref{eq:Calderon_LI}, together with Cauchy-Schwarz'
inequality, yields that
\[
\sum_{j \in J \; : \; \alpha \in \Gamma_j^{\perp}}
\frac{1}{\covol(\Gamma_j)} \int_{P_j} \abs{ \hat{g}_{j,p} (\cdot)
  \hat{g}_{j,p} (\cdot + \alpha) } \; d\mu_{P_j} (p) \in L^1_{\loc}
(\widehat{G} \setminus E),
\]
which in turn implies that
\[
\sum_{j \in J \; : \; \alpha \in \Gamma_j^{\perp}}
\int_{\widehat{G}} \biggl| \hat{f} (\omega) \hat{f}(\omega + \alpha)
\frac{1}{\covol(\Gamma_j)} \int_{P_j} \hat{g}_{j,p} (\omega)
\hat{g}_{j,p} (\omega + \alpha) \; d\mu_{P_j} (p) \; \biggr| \;
d\mu_{\widehat{G}} (\omega) < \infty.
\]
Thus, by Fubini-Tonelli's theorem, the series and integral defining
$c_{\alpha}$ can be interchanged to obtain the desired form
\eqref{eq:Fourier_coefficients}.
\end{proof} 
\end{proposition}

To ease notation, we define the following functions appearing
implicitly in the Fourier coefficients in Proposition
\ref{prop:GTIgeneralisedfourierseries}.
\begin{definition}
  Let
  $\cup_{j \in J} \{T_{\gamma} g_{j,p} \}_{\gamma \in \Gamma_j, p \in
    P_j}$
  be a generalized trans\-la\-tion-in\-vari\-ant system satisfying
\begin{align} \label{eq:Calderon_bounded}
\sum_{j \in J} \frac{1}{\covol(\Gamma_j)} \int_{P_j} \abs{\hat{g}_{j,p} (\omega) }^2 \; d\mu_{P_j} (p) < \infty
\end{align}
for $\mu_{\widehat{G}}$-a.e. $\omega \in \widehat{G}$. For the system $\cup_{j \in J} \{T_{\gamma} g_{j,p} \}_{\gamma \in \Gamma_j, p \in P_j}$, the associated \emph{auto-correlation functions} $ t_{\alpha}$, $\alpha \in \cup_{j \in J} \Gamma_j^{\perp}$ are $\mu_{\widehat{G}}$-a.e. defined by
\[
t_{\alpha} : \widehat{G} \to \mathbb{C}, \; \omega \mapsto \sum_{ j \in J \; : \; \alpha \in \Gamma_j^{\perp} } \frac{1}{\covol(\Gamma_j)} \int_{P_j}  \hat{g}_{j,p} (\omega) \overline{\hat{g}_{j,p} (\omega + \alpha)} \;  d\mu_{P_j} (p).
\]
\end{definition}

Phrased in terms of the auto-correlation functions, the
assumptions~\eqref{eq:Calderon_LI} and \eqref{eq:Calderon_bounded}
require $t_0 : \widehat{G} \to \mathbb{C}$ to be locally integrable
and uniformly bounded, respectively. Any generalized
trans\-la\-tion-in\-vari\-ant system forming a Bessel family with
upper bound $B > 0$ satisfies
\begin{equation} 
  \label{eq:t0-bounded-calderon}
  t_0(\omega) \le B
\end{equation}
for $\mu_{\widehat{G}}$-a.e. $\omega \in \widehat{G}$, as shown in
\cite{MR2283810,MR1916862,JakobsenReproducing2014}. The assumptions of
any of our results in Section~\ref{sec:suff_nec} imply
\eqref{eq:Calderon_bounded}, hence the auto-correlation functions
$t_\alpha$ will always be well-defined. In particular,
assumption~\eqref{eq:int-over-Pj-loc-int} is satisfied.

\section{Sufficient and necessary conditions for the frame
  property} \label{sec:suff_nec} This section contains the main
results of the paper. Section ~\ref{sec:necessary-conditions} provides
a necessary condition for the Bessel and frame property of a
generalized trans\-la\-tion-in\-vari\-ant system. Sufficient
conditions for the Bessel and frame properties of generalized
trans\-la\-tion-in\-vari\-ant systems are presented in
Section~\ref{sec:main-result}.  The obtained sufficient conditions are
compared to known frame bound estimates in
Section~\ref{sec:cc-cond-absol}.

\subsection{Necessary conditions}
\label{sec:necessary-conditions}

Inequality~\eqref{eq:t0-bounded-calderon} is a necessary condition for
the Bessel property of a generalized trans\-la\-tion-in\-vari\-ant
system. Under a weak regularity assumption, Theorem
\ref{thm:GTInec_little-l2} below presents a much stronger necessary
condition for a generalized trans\-la\-tion-in\-vari\-ant system to
form a Bessel family.

The proof of Theorem \ref{thm:GTInec_little-l2} makes use of a
differentiation process for integrals on locally compact groups as in
\cite[Section 2]{edwards1965pointwise}. In order to apply this, the
following notion is useful, cf. \cite[Definition 2.1]{edwards1965pointwise}. 

\begin{definition}
  Let $G$ be a locally compact group with Haar measure $\mu_{G}$. A
  decreasing sequence $(U_k )_{k \in \mathbb{N}}$ of finite measure
  Borel sets is called a \emph{D'-sequence} in $G$ if every
  neighborhood of $0$ contains some $U_k$, and if there exists a
  constant $C > 0$ such that
\[
0 < \mu_{G} (U_k - U_k) \leq C\mu_G (U_k)
\]
for all $k \in \mathbb{N}$, where $U_k - U_k := \setprop{u-v}{u,v\in U_k}$.
\end{definition}

Given a $D'$-sequence $(U_k)_{k \in \mathbb{N}}$ for $G$ and an $f \in
L^1_{\mathrm{loc}}(G)$, a point $x_0 \in G$ satisfying
\begin{equation}
\lim_{k \to \infty}  \frac{1}{\mu_{G} (U_k)} \int_{x_0
  + U_k} f(x) \; d\mu_{G} (x)  =  f(x_0)
\label{eq:leb-diff-thm}
\end{equation}
is called a \emph{Lebesgue point} of $f$. Lebesgue's differentiation
theorem \cite[Theorem 44.18]{MR0262773} asserts that the set of
Lebesgue points of any $f \in L^1_{\mathrm{loc}}(G)$ has full measure,
or equivalently, that \eqref{eq:leb-diff-thm} holds for a.e.
$x_0 \in G$.

In \cite{BowRos2014}, it is shown that any compactly generated abelian
Lie group $G$ admits a $D'$-sequence. As a consequence, any locally
compact abelian group $G$ of the form
$G = \mathbb{R}^d \times \mathbb{T}^m \times \mathbb{Z}^n \times F$,
where $d,m,n \in \mathbb{N}$ and $F$ is finite, possesses a
$D'$-sequence. On the other hand, Bownik and Ross~\cite{BowRos2014}
also show that some infinite dimensional locally compact abelian
groups, e.g., the tubby torus $\T^{\aleph_0}$, do not admit a
$D'$-sequence.

\begin{theorem} \label{thm:GTInec_little-l2} Let $G$ be a locally
  compact abelian group such that $\widehat{G}$ admits a
  $D'$-sequence.  Let
  $\cup_{j \in J} \{T_{\gamma} g_{j,p} \}_{\gamma \in \Gamma_j, p \in
    P_j}$
  be a generalized trans\-la\-tion-in\-vari\-ant system satisfying the
  $1$-UCP such that the Fourier series of $w_f$ converges
  unconditionally pointwise to $w_f(x_0)$ for some $x_0\in G$. Suppose
  $\cup_{j \in J} \{ T_{\gamma} g_{j, p} \}_{\gamma \in \Gamma_j, p
    \in P_j}$
  forms a Bessel system in $L^2 (G)$ with Bessel bound $B$. Then
\begin{align}\label{eq:B_1-lower}
 B &\ge B_2:= \esssup_{\omega \in \widehat{G}} \left(\sum_{\alpha \in
  \bigcup_{j \in J} \Gamma_j^{\perp}}
       \abs{t_{\alpha}(\omega)}^2\right)^{1/2}.
 \end{align}
Moreover, if $\cup_{j \in J} \{ T_{\gamma} g_{j, p} \}_{\gamma \in
  \Gamma_j, p \in P_j}$, in addition, forms a frame with lower frame bound
$A > 0$, then $A_{\infty}:=\essinf_{\omega \in \widehat{G}} t_0(\omega) \ge A$.
\end{theorem}

\begin{proof}
The ``moreover''-part is a consequence of~\cite[Theorem 3.13]{bandwidth}. 
The remainder of the proof is divided into three steps:

\textbf{Step 1:} \emph{Rewriting the frame operator.} Let
$S : L^2 (G) \to L^2 (G)$ denote the frame operator associated with
the Bessel family
$\cup_{j \in J} \{T_{\gamma} g_{j,p} \}_{\gamma \in \Gamma_j, p \in
  P_j}$.
The map $ (f_1, f_2) \mapsto \langle S f_1, f_2 \rangle$ is a
well-defined, bounded sesquilinear form on $L^2 (G) \times L^2 (G)$
with
 \begin{equation}
 |\langle Sf_1, f_2 \rangle | \leq B \| f_1 \|_2 \| f_2 \|_2
  \label{eq:frame_operator_bessel}
 \end{equation}
 for all  $f_1, f_2 \in L^2 (G)$. The
frame operator is related to the function $w_f : G \to \mathbb{C}$ introduced in
\eqref{eq:UCP} by $w_f(x)=\innerprod{ST_xf}{T_xf}$, in particular,
$w_f(0)=\innerprod{Sf}{f}$ for $f \in \mathcal{D}_{E} (G)$.  By
translation invariance of $\mathcal{D}_E (G)$, we can assume $x_0=0$, that
is, pointwise convergence of the Fourier series of $w_f$ at the origin to $w_f(0)$.
By Proposition~\ref{prop:GTIgeneralisedfourierseries}, it then follows that 
\begin{align} \label{eq:Walnut}
\innerprod{Sf}{f} = \sum_{\alpha \in \bigcup_{j \in J} \Gamma_j^{\perp}}
  \int_{\widehat{G}} \hat{f} (\omega) \overline{\hat{f} (\omega +
    \alpha)} \overline{t_{\alpha} (\omega)} \; d\mu_{\widehat{G}} (\omega)
\end{align}
for all $f \in \mathcal{D}_{E} (G)$. Here, each auto-correlation function
 $t_{\alpha} : \widehat{G} \to \mathbb{C}$ is well-defined with
 $\|t_{\alpha}\|_{\infty} \leq B$ by \eqref{eq:t0-bounded-calderon}.
The identity \eqref{eq:Walnut}, together with an application of the polarization identity for sesquilinear forms
  and the bound \eqref{eq:frame_operator_bessel}, therefore gives
  \begin{align}
 \label{eq:K_convergence}
    \absBig{ \sum_{\alpha \in
    \bigcup_{j \in J} \Gamma_j^{\perp}} \int_{\widehat{G}}
    \hat{f}_1 (\omega) \overline{\hat{f}_2 (\omega + \alpha)}
    \overline{t_{\alpha} (\omega)} \; d\mu_{\widehat{G}} (\omega)
    \,} \leq B \|f_1 \|_2 \|f_2 \|_2
\end{align}
for all $f_1, f_2 \in \mathcal{D}_{E} (G)$. For clarity, 
we define $c_{\alpha}:=\int_{\widehat{G}} \hat{f}_1 (\omega)
\overline{\hat{f}_2 (\omega + \alpha)} \overline{t_{\alpha} (\omega)}
d\mu_{\widehat{G}} (\omega)$ for $\alpha \in \Lambda :=  \bigcup_{j
  \in J} \Gamma_j^{\perp}$. Then \eqref{eq:K_convergence} simply reads
  $| \sum_{\alpha} c_{\alpha} | \leq B \| f_1 \|_2 \| f_2 \|_2$. 

  \textbf{Step 2:} \emph{Construction of test functions $f_1, f_2$.}
  First, we assume that $E=\emptyset$ in the $1$-UCP assumption. Let
  $\omega_0 \in \widehat{G}$ be a common Lebesgue point of
  $\abs{t_\alpha}^2\in L^\infty(\ghat)\subset
  L^1_{\mathrm{loc}}(\ghat)$
  for all $\alpha \in \Lambda$. The dual group $\widehat{G}$ is second
  countable, whence metrizable.  The metric $d_{\widehat{G}}$ inducing
  the given topology on $\widehat{G}$ can be chosen to be
  translation-invariant. Let $\sigma: \N \to J$ be a bijection and
  define
  $\Lambda_{m,n}:=\bigcup_{i = 1}^n \Gamma_{\sigma (i)}^{\perp} \cap
  B(0,m)$,
  where $B(0,m)$ denotes the open ball, relative to $d_{\widehat{G}}$,
  with radius $m > 0$ and center $0 \in \ghat$. Then, given any
  $\alpha \in \Lambda$, there exists $m,n \in \N$ such that
  $\alpha \in \Lambda_{m,n}$.

  For fixed $m,n \in \N$, set
  $\delta_{m,n}:=\min\setpropsmall{d_{\widehat{G}} (\alpha,\alpha')}{
    \alpha, \alpha', \in \Lambda_{m,n} \text{ with } \alpha \neq
    \alpha'}$.
  Note that $\delta_{m,n}>0$ since $\Lambda_{m,n}$ is a finite set.
  Let $(U_k)_{k \in \mathbb{N}}$ be a $D'$-sequence in $\widehat{G}$.
  The sets $U_k$ lie eventually inside an arbitrary neighborhood of
  $0 \in \ghat$. Thus, by local compactness of $\ghat$, we can assume
  without loss of generality that $U_1$ is relatively
  compact. Moreover, we let $K \in \mathbb{N}$ be so that
  $U_k \subset B(0,\delta_{m,n}/2)$ for all $k \ge K$. Then, for all
  $k \ge K$,
\begin{equation}
  \mu_{\ghat}((\alpha + U_k) \cap (\alpha' + U_k)) = 0 \quad \text{
    for all $\alpha, \alpha' \in \Lambda_{m,n}$ with $\alpha \neq
    \alpha'$}.
\label{eq:Dprime-seq-alpha}
\end{equation}
Define $f_1 \in \mathcal{D}_{\emptyset} (G)$ by
$\hat{f}_1 := \mu_{\widehat{G}} (U_k)^{-1/2}
\mathds{1}_{\omega_0+U_k}$.
For $k \ge K$, define $h : \widehat{G} \to \mathbb{C}$ on
$\omega_0+U_k+\Lambda_{m,n}$ by
\[
h (\omega + \alpha) =
\overline{t_{\alpha} (\omega)} \qquad \text{for a.e. } \omega \in \omega_0+U_k
\]
for each $\alpha \in \Lambda_{m,n}$ and by $h(\omega)=0$ for
$\omega \in \ghat \setminus (\omega_0+U_k+\Lambda_{m,n})$.  The
property \eqref{eq:Dprime-seq-alpha} of $U_k$ guarantees that $h$ is
well-defined. Let $\hat{f}_2 := \| h \|_{2}^{-1} h$. Then
$f_2 \in \mathcal{D}_{\emptyset} (G)$ with $\| f_2 \|_2 = 1$. A direct
calculation entails
\begin{align*}
\sum_{\alpha \in \Lambda_{m,n}} c_{\alpha} &= \sum_{\alpha \in \Lambda_{m,n}} \int_{\widehat{G}} \hat{f}_1 (\omega) \overline{\hat{f}_2 (\omega + \alpha)} \overline{t_{\alpha} (\omega)} \; d\mu_{\widehat{G}} (\omega) \\
&= \| h \|_2^{-1} \sum_{\alpha \in \Lambda_{m,n}} \frac{1}{\mu_{\widehat{G}} (U_k)^{1/2}} \int_{\omega_0+U_k} |t_{\alpha} (\omega) |^2 \; d\mu_{\widehat{G}} (\omega) \\
&= \frac{1}{\mu_{\widehat{G}} (U_k)^{1/2}} \bigg(  \sum_{\alpha \in \Lambda_{m,n}}  \int_{\omega_0+U_k} |t_{\alpha} (\omega) |^2 \; d\mu_{\widehat{G}} (\omega) \bigg)^{-1/2} \bigg(  \sum_{\alpha \in \Lambda_{m,n}}  \int_{\omega_0+U_k} |t_{\alpha} (\omega) |^2 \; d\mu_{\widehat{G}} (\omega) \bigg) \\
&= \bigg( \sum_{\alpha \in \Lambda_{m,n}} \frac{1}{\mu_{\widehat{G}} (U_k)} \int_{\omega_0+U_k}  |t_{\alpha} (\omega) |^2 \; d\mu_{\widehat{G}} (\omega) \bigg)^{1/2}
\end{align*}
for any $m,n \in \N$ and $k \ge K$.  An application of Lebesgue's
differentiation theorem \cite[Theorem 44.18]{MR0262773} next gives
\[
\lim_{k \to \infty} \bigg( \sum_{\alpha \in \Lambda_{m,n}}
\frac{1}{\mu_{\widehat{G}} (U_k)} \int_{\omega_0 + U_k} |t_{\alpha}
(\omega) |^2 \; d\mu_{\widehat{G}} (\omega) \bigg)^{1/2} = \bigg(
\sum_{\alpha \in \Lambda_{m,n}} |t_{\alpha} (w_0) |^2 \; \bigg)^{1/2}
\]
for any $m,n \in \N$.

\textbf{Step 3:} \emph{An $\eps$-argument.} For arbitrary $\eps > 0$, there exist $M, N \in \mathbb{N}$ such that, for
 all $m \ge M, n \ge N$,
\begin{align} \label{eq:distance_epsilon}
\bigg| \sum_{\alpha \in \Lambda} c_{\alpha} - \sum_{\alpha \in \Lambda_{m,n}} c_{\alpha}   \bigg| \leq \eps.
\end{align}
Using \eqref{eq:K_convergence} and \eqref{eq:distance_epsilon}, it
now follows by the triangle inequality that 
\[
 \bigg( \sum_{\alpha \in \Lambda_{m,n}} |t_{\alpha} (w_0) |^2 \;
 \bigg)^{1/2} \leq B + \eps
\]
for $m \ge M$ and $n \ge N$.  Since
$\sum_{\alpha \in \Lambda_{m,n}} |t_{\alpha} (w_0) |^2$ is a bounded
and monotone increasing sequence of $n$ and $m$, the limit
\[
 \bigg( \sum_{\alpha \in \Lambda} |t_{\alpha} (w_0) |^2 \;
 \bigg)^{1/2} = \sup_{m\ge M,n\ge N}  \bigg( \sum_{\alpha \in \Lambda_{m,n}} |t_{\alpha} (w_0) |^2 \;  \bigg)^{1/2} 
\]
exists. 

Since $\eps > 0$ was taken arbitrary, it follows that 
\begin{equation}
 \bigg( \sum_{\alpha
  \in \Lambda} |t_{\alpha} (\omega_0) |^2 \;  \bigg)^{1/2}\leq
B.
\label{eq:little-l2-bound-at-w0}
\end{equation}
Being a countable union of null sets, the complement of the set of
common Lebesgue points of $\abs{t_\alpha}^2$, $\alpha \in \Lambda$, is
a null set. We conclude that \eqref{eq:little-l2-bound-at-w0} holds
for $\mu_{\widehat{G}}$-a.e. $\omega_0 \in \ghat$, which completes the
proof for the case $E = \emptyset$.

Lastly, consider the case that
$\cup_{j \in J} \{T_{\gamma} g_{j,p} \}_{\gamma \in \Gamma_j, p \in
  P_j}$
satisfies the $1$-UCP with respect to an arbitrary
$E \in \mathcal{E}$. In this case the Fourier supports of $f_1$ and
$f_2$ might intersect $E$. If this happens, the functions $f_1$ and
$f_2$ should be approximated by functions from $\mathcal{D}_E (G)$,
see \cite[Remark 4]{JakobsenReproducing2014} for the details.
\end{proof}

\begin{remark}
  \begin{enumerate}[(i)]\item 
    In Theorem~\ref{thm:GTInec_little-l2}, the assumption of $1$-UCP
    with pointwise unconditionally convergence of the Fourier series
    of $w_f$ can be replaced by the simpler, but stronger, assumption
    that
    $\cup_{j \in J} \{T_{\gamma} g_{j,p} \}_{\gamma \in \Gamma_j, p
      \in P_j}$ satisfies that $\infty$-UCP.
\item While the existence of a $D'$-sequence in $\ghat$ is sufficient
  for the differentiation 
process on Lebesgue's integrals \cite[Theorem 44.18]{MR0262773}, it
may not be necessary. In fact, it is an open problem whether
Lebesgue's differentiation theorem holds on all second countable
locally compact abelian groups \cite[Section 7.1]{BowRos2014}. 
\end{enumerate}

\end{remark}

The $1$-UCP assumption in Theorem~\ref{thm:GTInec_little-l2} cannot be
dropped as demonstrated by Example~\ref{ex:bownik_example1} below. The
construction of the generalized shift-invariant system follows \cite[Example 3.2]{MR2066821}.

\begin{example} \label{ex:bownik_example1}
Let $G = \mathbb{Z}$. Let $N \in \mathbb{N}$
be such that $N \geq 2$. Define the lattices $\Gamma_j = N^j
\mathbb{Z}$ for $j \in \N$. Let $\tau_1=0$ and define $\tau_j$, $j \ge
2$, inductively as the smallest $t \in \Z$ in absolute value satisfying 
\begin{align} \label{eq:induction}
\left(\bigcupdot_{i=1}^{j-1} \left(\tau_i + N^i \Z \right)\right) \cap \left(t +
  N^j \Z\right) = \emptyset.
\end{align}
In case $t$ and $-t$ both are minimizers,  pick $\tau_j$ to be
positive. Then 
\[
 \Z = \bigcupdot_{j \in \N} \bigg( \tau_j + N^j \Z \bigg), 
\]
with the union being disjoint, see also \cite[Lemma
4.4]{bandwidth}. Define the generators
$g_j = \mathds{1}_{\tau_j} \in \ell^2 (\mathbb{Z})$ for
$j \in \mathbb{N}$. By construction, the generalized shift-invariant
system
$\cup_{j \in \mathbb{N}} \{ T_{\gamma} g_j \}_{\gamma \in \Gamma_j}$
is the canonical basis $\{\mathds{1}_{k} \}_{k \in \mathbb{Z}}$ for
$\ell^2 (\mathbb{Z})$ and thus it forms, in particular, a frame for
$\ell^2 (\mathbb{Z})$ with frame bounds $A = B = 1$. The system
$\cup_{j \in \mathbb{N}} \{ T_{\gamma} g_j \}_{\gamma \in \Gamma_j}$
satisfies the $1$-UCP only for the case $N = 2$ as shown in
\cite[Example 1]{bandwidth}. We now show that the bound
\eqref{eq:B_1-lower} fails for $N \ge 3$ in spite of the Bessel
property.

Observe that $\hat{g}_j \in L^2 (\mathbb{T})$ is given by $\hat{g}_j (\omega) = e^{2 \pi i \tau_j \omega}$
for $\omega \in \itvco{0}{1}$. Hence, for any $\alpha \in \bigcup_{j=1}^\infty
N^{-j}\Z$,
\[ 
t_\alpha(\omega)=\sum_{  j \in \N \; : \; \alpha \in N^{-j}\Z } \frac{1}{\covol
  (\Gamma_j)} \hat{g}_j (\omega) \overline{\hat{g}_j (\omega +
  \alpha)} = \sum_{  j \in \N \; : \; \alpha \in N^{-j}\Z } N^{-j} \myexp{-2 \pi i \alpha
  \tau_j} 
\]
for all $\omega \in \itvco{0}{1}$. Since $t_\alpha$ is independent of the variable $\omega \in
\itvco{0}{1}$, we fix an arbitrary $\omega \in \itvco{0}{1}$ in the following calculations. 

We rewrite $\bigcup_{j=1}^\infty
N^{-j}\Z$ as the following disjoint union:
\[ \bigcup_{j=1}^\infty
N^{-j}\Z = \mathbb{Z} \bigcupdot \bigg( \bigcupdot_{m=1}^\infty
N^{-m}(\Z\setminus N\Z) \bigg). \]
For $\alpha \in \mathbb{Z}$, a direct calculation gives
\[
t_{\alpha} (\omega) = \sum_{j = 1}^{\infty} N^{-j} = \frac{1}{N-1}.
\]
On the other hand, writing $\alpha \notin \mathbb{Z}$ as $\alpha = k N^{-m}$ for $k \in \Z \setminus
N\Z$ and $m \in \N$ yields
\[ 
t_\alpha(\omega) = \sum_{j=m}^\infty N^{-j} \myexp{-2 \pi i
  k N^{-m} \tau_j}.
\]
Hence
\[ 
  \abs{t_\alpha(\omega)} = N^{-m} \abs{\sum_{\ell=0}^\infty N^{-\ell} \myexp{-2 \pi i
  k N^{-m} \tau_{\ell+m}}}.
\]
Next, we claim that it then follows that
\[ \abs{t_\alpha(\omega)}  \ge N^{-m} \left(1-\sum_{\ell=1}^\infty N^{-\ell}\right) =  N^{-m}\,
\frac{N-2}{N-1} .\]
To see this claim, let $z_\ell \in \T$, $\ell \in \mathbb{N}_0$. By the triangle inequality, we have
\begin{multline*}
  \abs{\sum_{\ell=0}^\infty N^{-\ell} z_\ell} \ge \abs{z_0} - \abs{z_0-\sum_{\ell=0}^\infty N^{-\ell} z_\ell} = 1 -
\abs{\sum_{\ell=1}^\infty N^{-\ell} z_\ell} \ge 1 -
\sum_{\ell=1}^\infty N^{-\ell} =\frac{N-2}{N-1},
\end{multline*}
which proves the claim.

Let $N\ge 3$. Since for each $m \in \N$ there are infinitely many $k \in \Z\
\setminus N\Z$ with $\abs{t_{kN^{-m}}(\omega)}>N^{-m}/2$, we see
that \eqref{eq:B_1-lower} is violated as $B_2=\infty$. In fact, for
any $p \in \itvco{1}{\infty}$, we have
\[ 
\sum_{\alpha \in \bigcup_{j \in J} \Gamma_j^{\perp}}
\abs{t_\alpha(\omega)}^p = \infty 
\]
for all $\omega \in \itvco{0}{1}$. Thus, no $\ell^p$-norm of
$\{t_\alpha(\omega)\}_{\alpha \in \bigcup_{j \in J}
  \Gamma_j^{\perp}}$, with $\omega \in \itvco{0}{1}$, can be finite.
\end{example}

The discreteness of the group $G=\Z$ in
Example~\ref{ex:bownik_example1} is not crucial. In fact, the
construction is easily transferred to $L^2(\R)$, e.g., by starting
from the Gabor-like orthonormal basis $\{ T_k \myexp{2\pi i m \cdot}
\mathds{1}_{[0,1)} \}_{k,m \in \mathbb{Z}}$.   

\subsection{Sufficient conditions} 
\label{sec:main-result}
The following result, Theorem~\ref{thm:GTIsufficient_CC}, provides a
sufficient condition and estimates of the frame bounds for generalized
trans\-la\-tion-in\-vari\-ant frames. The proof of
Theorem~\ref{thm:GTIsufficient_CC} is based on the simple estimate
that a (generalized) Fourier series of an almost periodic function is
bounded from above by the sum of the modulus of its coefficients and
is bounded from below by the absolute value of its constant term minus
the sum of the other terms in modulus.

\begin{theorem} \label{thm:GTIsufficient_CC}
Let $\cup_{j \in J} \{T_{\gamma} g_{j,p} \}_{\gamma \in \Gamma_j, p \in P_j}$ be a generalized trans\-la\-tion-in\-vari\-ant system satisfying the $1$-UCP.
\begin{enumerate}[(i)]
\item Suppose $\cup_{j \in J} \{ T_{\gamma} g_{j, p} \}_{\gamma \in \Gamma_j, p \in P_j}$ satisfies
\begin{align} \label{eq:B_cc-t-alpha}
 B_1 := \esssup_{\omega \in \widehat{G}} \sum_{\alpha \in \bigcup_{j \in J} \Gamma_j^{\perp}} |t_{\alpha} (\omega)|  < \infty,
 \end{align}
then $\cup_{j \in J} \{ T_{\gamma} g_{j, p} \}_{\gamma \in \Gamma_j, p \in P_j}$ forms a Bessel system in $L^2 (G)$ with Bessel bound $B_1$.
\item Suppose $\cup_{j \in J} \{ T_{\gamma} g_{j, p} \}_{\gamma \in
    \Gamma_j, p \in P_j}$ satisfies \eqref{eq:B_cc-t-alpha} and 
\begin{align} \label{eq:A_cc-t-alpha}
 A_1 := &\essinf_{\omega \in \widehat{G}} \bigg(t_{0} (\omega) -  \sum_{\alpha \in \bigcup_{j \in J} \Gamma_j^{\perp} \setminus \{0\}} |t_{\alpha} (\omega) | \bigg)> 0, 
 \end{align}
then $\cup_{j \in J} \{ T_{\gamma} g_{j, p} \}_{\gamma \in \Gamma_j, p
  \in P_j}$ forms a frame for $L^2 (G)$ with lower bound $A_1$ and
upper bound $B_1$.
\end{enumerate}
\end{theorem}
\begin{proof} 
  Suppose the system
  $\cup_{j \in J} \{ T_{\gamma} g_{j, p} \}_{\gamma \in \Gamma_j, p
    \in P_j}$
  satisfies condition \eqref{eq:B_cc-t-alpha} and the 1-UCP with
  respect to $E \in \mathcal{E}$. By definition of $w_f$ and the fact
  that $\mathcal{D}_E (G)$ is dense in $L^2(G)$, it suffices in (ii)
  to show that $A_1 \norm{f}^2 \le w_f(0) \le B_1 \norm{f}^2$ for all
  $f \in \cD_E(G)$, while in (i) it suffices to prove the upper bound.

  As a consequence of
  Proposition~\ref{prop:GTIgeneralisedfourierseries}, the auxiliary
  function $w_f : G \to \mathbb{C} $ possesses the generalized Fourier
  series
  $\sum_{\alpha \in \bigcup_{j \in J} \Gamma_j^{\perp}} c_\alpha
  \alpha$,
  where 
  $c_\alpha:=\int_{\widehat{G}} \hat{f} (\omega) \overline{\hat{f}
    (\omega + \alpha)} \overline{t_{\alpha} (\omega)} \;
  d\mu_{\widehat{G}} (\omega)$.
  It will first be shown that $w_f$ coincides point-wise with its
  generalized Fourier series. In order to do this, we show that the
  generalized Fourier series is uniformly convergent. An application
  of Beppo Levi's theorem and Young's inequality for products gives
\begin{align}
\sum_{\alpha \in \bigcup_{j \in J}
        \Gamma_j^{\perp} \setminus \{0\}} \abs{ c_\alpha} \nonumber 
&\leq \int_{\widehat{G}} \sum_{\alpha \in \bigcup_{j \in J}
         \Gamma_j^{\perp} \setminus \{0\}} |\hat{f} (\omega) \hat{f}
         (\omega + \alpha) t_{\alpha} (\omega) | \; d\mu_{\widehat{G}}
         (\omega) \nonumber \\
&\leq
\frac{1}{2} \int_{\widehat{G}}| \hat{f} (\omega) |^2 \sum_{\alpha \in
  \bigcup_{j \in J} \Gamma_j^{\perp} \setminus \{0\}}  | t_{\alpha}
  (\omega) | \; d\mu_{\widehat{G}} (\omega) \nonumber \\
& \quad \quad \quad \quad \quad \quad \quad + \frac{1}{2}
  \int_{\widehat{G}} \sum_{\alpha \in \bigcup_{j \in J}
  \Gamma_j^{\perp} \setminus \{0\}} |\hat{f} (\omega + \alpha) |^2 |
  t_{\alpha} (\omega) | \; d\mu_{\widehat{G}} (\omega) \nonumber \\
&=  \int_{\widehat{G}}| \hat{f} (\omega) |^2 \sum_{\alpha \in
  \bigcup_{j \in J} \Gamma_j^{\perp} \setminus \{0\}}  \abs{
  t_{\alpha} (\omega) } \; d\mu_{\widehat{G}} (\omega), \label{eq:Rf-estimate}
\end{align}
where equality follows by the change of variable
$\omega \mapsto \omega - \alpha$ and
$ \overline{t_{\alpha} (\omega - \alpha)} = t_{- \alpha}
(\omega)$.
Since
$\sum_{\alpha \in \bigcup_{j \in J} \Gamma_j^{\perp}} |c_{\alpha}| <
\infty $,
an application of the Weierstrass M-test yields that the generalized
Fourier series of $w_f$ converges uniformly to an almost periodic
function. By uniqueness of Fourier coefficients, it follows that
$w_f (x) = \sum_{\alpha \in \bigcup_{j \in J} \Gamma_j^{\perp}}
c_{\alpha} \alpha (x)$ pointwise for all $x \in G$.

Now, setting $x = 0$ in the Fourier series representation of $w_f$ and using
\eqref{eq:Rf-estimate} give
\begin{align*}
                   w_f(0) &= \sum_{\alpha \in \bigcup_{j \in J}
    \Gamma_j^{\perp} } c_\alpha \leq \int_{\widehat{G}} |\hat{f}(\omega)|^2 \sum_{\alpha \in \bigcup_{j \in J} \Gamma_j^{\perp}} |t_{\alpha} (\omega)| \; d\mu_{\widehat{G}} (\omega) \leq B_1 \| f \|^2_2
\end{align*}
for all $f \in \mathcal{D}_E (G)$. This shows (i). Assume now also
that the assumption in (ii) is satisfied. Then, by the triangle
inequality and \eqref{eq:Rf-estimate}, we have for all $f \in \cD_E(G)$

\begin{align*}
  w_f(0) &= \sum_{\alpha \in \bigcup_{j \in J}
    \Gamma_j^{\perp} } c_\alpha  
    \geq c_0 -  \absBig{
\sum_{\alpha \in \bigcup_{j \in J}    \Gamma_j^{\perp} \setminus
\{0\}} c_\alpha\, }  \\
&\ge \int_{\ghat}
\abssmall{\hat{f}(\omega)}^2 \biggl( t_0(\omega)- \sum_{\alpha \in \bigcup_{j \in J}    \Gamma_j^{\perp} \setminus
\{0\} } \abs{t_\alpha(\omega)} \biggr) d\mu_{\widehat{G}} (\omega) \ge
A_1 \| f \|^2_2
\end{align*}
as desired.
\end{proof}

The frame bound estimates of Theorem~\ref{thm:GTIsufficient_CC} are
optimal for tight frames. That is, for a generalized
trans\-la\-tion-in\-vari\-ant system
$\cup_{j \in J} \{ T_{\gamma} g_{j, p} \}_{\gamma \in \Gamma_j, p \in
  P_j}$
that satisfies the 1-UCP and forms a tight frame, the estimates in
Theorem~\ref{thm:GTIsufficient_CC} recover precisely the frame bound
of the given frame. This simple observation is stated as the next
result.

\begin{proposition} \label{prop:tight_CC}
Let $\cup_{j \in J} \{ T_{\gamma} g_{j, p} \}_{\gamma \in \Gamma_j, p
  \in P_j}$ be a generalized trans\-la\-tion-in\-vari\-ant system satisfying
the 1-UCP. Suppose $\cup_{j \in J} \{ T_{\gamma} g_{j, p}
\}_{\gamma \in \Gamma_j, p \in P_j}$ forms a tight frame for $L^2 (G)$
with frame bound $A > 0$. Then $A = A_1 = B_1$. 
 \begin{proof}
   Suppose the system
   $\cup_{j \in J} \{ T_{\gamma} g_{j, p} \}_{\gamma \in \Gamma_j, p
     \in P_j}$
   is a tight frame for $L^2 (G)$ with bound $A > 0$. By~\cite[Theorem
   3.4]{JakobsenReproducing2014} and~\cite[Theorem 3.11]{bandwidth},
   it holds, for any $\alpha \in \bigcup_{j \in J} \Gamma_j^{\perp}$,
   that
 \[ t_{\alpha} (\omega) = A \delta_{\alpha, 0} \]
 for $\mu_{\widehat{G}}$-a.e. $\omega \in \widehat{G}$. Hence $\sum_{\alpha \in
  \bigcup_{j \in J} \Gamma_j^{\perp} \setminus \{0\}}
\abs{t_{\alpha}(\omega)}=0$ almost everywhere on $\widehat{G}$ and the conclusion follows. 
 \end{proof}
\end{proposition}

\subsection{Comparison of frame bound estimates}
\label{sec:cc-cond-absol}
In this section we compare the frame bound estimates provided by
Theorem~\ref{thm:GTIsufficient_CC} with known estimates. For this, we
state the following result~\cite[Proposition
3.7]{JakobsenReproducing2014}.

\begin{proposition} \label{prop:GTI_sufficient_absCC}
Let $\cup_{j \in J} \{ T_{\gamma} g_{j, p} \}_{\gamma \in \Gamma_j, p \in
  P_j}$ be a generalized trans\-la\-tion-in\-vari\-ant system.
\begin{enumerate}[(i)]
\item Suppose the system $\cup_{j \in J} \{ T_{\gamma} g_{j, p} \}_{\gamma \in \Gamma_j, p \in
  P_j}$ satisfies 
\begin{align} \label{eq:GTI_absCC}
 B' := \esssup_{\omega \in \widehat{G}} \sum_{j \in J} \frac{1}{\covol (\Gamma_j)} \int_{P_j} \sum_{\alpha \in \Gamma_j^{\perp}}  | \hat{g}_{j,p} (\omega)  \hat{g}_{j,p} (\omega +\alpha) |  d\mu_{P_j} (p)< \infty,
 \end{align}
 then
 $\cup_{j \in J} \{ T_{\gamma} g_{j, p} \}_{\gamma \in \Gamma_j, p \in
   P_j}$ forms a Bessel family in $L^2 (G)$ with Bessel bound $B'$.
\item Suppose the system $\cup_{j \in J} \{ T_{\gamma} g_{j, p} \}_{\gamma \in \Gamma_j, p \in
  P_j}$ satisfies and 
\begin{align*} 
 A' := &\essinf_{\omega \in \widehat{G}} \bigg(\sum_{j \in J} \frac{1}{\covol(\Gamma_j)} \int_{P_j} | \hat{g}_{j,p} (\omega)|^2 \; d\mu_{P_j} (p) \\
 & \quad \quad \quad \quad - \sum_{j \in J} \frac{1}{\covol(\Gamma_j)} \int_{P_j} \sum_{\alpha \in \Gamma_j^{\perp} \setminus \{0\}} | \hat{g}_{j,p} (\omega) \hat{g}_{j,p} (\omega + \alpha) | \; d\mu_{P_j} (p) \bigg)> 0,
 \end{align*}
 then
 $\cup_{j \in J} \{ T_{\gamma} g_{j, p} \}_{\gamma \in \Gamma_j, p \in
   P_j}$
 forms a frame for $L^2 (G)$ with lower bound $A'$ and upper bound
 $B'$.
\end{enumerate}
\end{proposition}

In~\cite{JakobsenReproducing2014}, the term \emph{absolute
  CC-con\-di\-tion} was used for condition \eqref{eq:GTI_absCC}.  The
important difference between the CC-con\-di\-tion
\eqref{eq:B_cc-t-alpha} and the absolute CC-con\-di\-tion
\eqref{eq:GTI_absCC} is the placement of the absolute sign in the
summand. In the CC-con\-di\-tion, it is possible to have phase
cancellations within each auto-correlation function while the absolute
CC-con\-di\-tion prohibits such cancellations. It is a simple
observation that a generalized trans\-la\-tion-in\-vari\-ant system
satisfying the absolute CC-con\-di\-tion also satisfies the
CC-con\-di\-tion.

\begin{lemma}
  Suppose
  $\cup_{j \in J} \{T_{\gamma} g_{j,p} \}_{\gamma \in \Gamma_j, p \in
    P_j}$
  satisfies the absolute CC-con\-di\-tion \eqref{eq:GTI_absCC}. Then
  $\cup_{j \in J} \{T_{\gamma} g_{j,p} \}_{\gamma \in \Gamma_j, p \in
    P_j}$ also satisfies the CC-con\-di\-tion \eqref{eq:B_cc-t-alpha}.
\begin{proof}
  Suppose the system
  $\cup_{j \in J} \{T_{\gamma} g_{j,p} \}_{\gamma \in \Gamma_j, p \in
    P_j}$
  satisfies the absolute CC-con\-di\-tion \eqref{eq:GTI_absCC}. Then
  an application of Beppo Levi's theorem gives
\begin{multline*}
  \sum_{j \in J} \sum_{\alpha \in \Gamma_j^{\perp}} \bigg|
  \frac{1}{\covol(\Gamma_j)} \int_{P_j}    \hat{g}_{j,p} (\omega)
  \overline{ \hat{g}_{j,p} (\omega + \alpha) } \; d\mu_{P_j} (p)
  \bigg| \\ \leq \sum_{j \in J} \frac{1}{\covol(\Gamma_j)} \int_{P_j} \sum_{\alpha \in \Gamma_j^{\perp}}  | \hat{g}_{j,p} (\omega) \hat{g}_{j,p} (\omega + \alpha) |  d\mu_{P_j} (p) 
   < \infty 
\end{multline*}
for $\mu_{\widehat{G}}$-a.e. $\omega \in \widehat{G}$. Using the
absolute convergence of the series, a re-ordering of the summation
does not affect the convergence. Thus
\[ 
 \esssup_{\omega \in \widehat{G}} \sum_{\alpha \in \bigcup_{j \in J}
   \Gamma_j^{\perp}} \bigg| \sum_{ j \in J \; : \; \alpha \in
   \Gamma_j^{\perp} } \frac{1}{\covol(\Gamma_j)}  \int_{P_j}
 \hat{g}_{j,p} (\omega) \overline{ \hat{g}_{j,p} (\omega + \alpha)  }
 \; d\mu_{P_j} (p) \bigg|  < \infty, 
\]
as required.
\end{proof}
\end{lemma}

Jakobsen and the first named author~\cite{JakobsenReproducing2014}
show that a generalized trans\-la\-tion-in\-vari\-ant system
satisfying the absolute CC-con\-di\-tion automatically satisfies the
$\alpha$-LIC.  Thus the $1$-UCP is implicitly assumed in the estimate
\eqref{eq:GTI_absCC}.

The generalized trans\-la\-tion-in\-vari\-ant system in Example
\ref{ex:bownik_example1} with $N = 2$ satisfies the frame bound
estimates based on the CC-con\-di\-tion, but dramatically fails the
estimates based on the absolute CC-con\-di\-tion as demonstrated in
the next example.

\begin{example} \label{ex:CC_ACC} Let $G=\mathbb{Z}$. Consider the
  system
  $\cup_{j \in \mathbb{N}} \{T_{\gamma} g_j \}_{\gamma \in \Gamma_j}$
  in $\ell^2 (\mathbb{Z})$ with $\Gamma_j = 2^j \mathbb{Z}$ and
  $g_j = \mathds{1}_{\tau_j}$, where
  $(\tau_j)_{j \in \mathbb{N}} \subset \mathbb{Z}$ is chosen as in
  \eqref{eq:induction}. This system forms a frame for
  $\ell^2 (\mathbb{Z})$ with frame bounds $A = B = 1$. For
  $\omega \in [0,1)$, a direct calculation gives
 \[
 \sum_{j \in \mathbb{N}} \frac{1}{\covol (\Gamma_j)} \sum_{\alpha \in
   \Gamma_j^{\perp}} |\hat{g}_j (\omega) \hat{g}_j (\omega + \alpha) |
 = \sum_{j \in \mathbb{N}} \frac{1}{2^j} \sum_{\alpha \in
   \Gamma_j^{\perp}} |e^{2 \pi i \tau_j \alpha} | = \sum_{j \in
   \mathbb{N}} \frac{1}{2^j} \# (\Gamma_j^{\perp}) = \sum_{j \in
   \mathbb{N}} 1 = \infty.
 \]
 Thus
 $\cup_{j \in \mathbb{N}} \{T_{\gamma} g_j \}_{\gamma \in \Gamma_j}$
 fails the estimate \eqref{eq:GTI_absCC}.  On the other hand, it
 follows by Proposition \ref{prop:tight_CC} that
 $\cup_{j \in \mathbb{N}} \{T_{\gamma} g_j \}_{\gamma \in \Gamma_j}$
 satisfies the estimates \eqref{eq:B_cc-t-alpha} and
 \eqref{eq:A_cc-t-alpha} with bounds $A_1 = B_1 = 1$.
\end{example}

The discrepancy between the frame bound estimates in Proposition
\ref{prop:GTI_sufficient_absCC} and the estimates in
Theorem~\ref{thm:GTIsufficient_CC} might occur even for well-known
orthonormal bases. Indeed, the Meyer wavelet is an example of an
orthonormal basis in $L^2 (\mathbb{R})$ for which the estimates based
on the absolute CC-con\-di\-tion give the poor estimates $A'=-1$ and
$B'=3$, see~\cite[p. 984]{Dau90}. However, the frame bound estimates
in Theorem \ref{thm:GTIsufficient_CC} give the correct frame bounds,
namely $A_1=B_1=1$. For a direct verification of the characterizing
equations $t_{\alpha} = \delta_{\alpha, 0}$ for the Meyer wavelet, the
interested reader to Daubechies' book \cite[Section 4.2.1]{Dau92}.

Finally, we remark that Casazza, Christensen and
Janssen~\cite{MR1814424} give an example of a Gabor system forming a
Bessel system in $L^2 (\mathbb{R})$, but where
$\sum_\alpha \abs{t_\alpha(\omega)}=\infty$ for a.e.\ $\omega \in
\R$.
This demonstrates that for Bessel systems both the CC-con\-di\-tion
and the absolute CC-con\-di\-tion can fail even though the LIC and
thus $\alpha$-LIC and the $1$-UCP hold.

\section{Applications and examples} 
\label{sec:examples}

In this section the sufficient conditions given in Theorem
\ref{thm:GTIsufficient_CC} will be considered for special types of
generalized trans\-la\-tion-in\-vari\-ant systems.  In the examples
presented in this section, the focus will be on explicit formulas for
the auto-correlation functions
$t_{\alpha} : \widehat{G} \to \mathbb{C}$ and the associated
\emph{remainder function}
\begin{equation*} 
R:\ghat \to \itvcc{0}{\infty}, \quad R(\omega)=\sum_{\alpha \in
  \bigcup_{j \in J} \Gamma_j^{\perp} \setminus \{0\}} \abs{t_{\alpha}(\omega)}, 
\end{equation*}
which are the main ingredients in the estimates in Theorem
\ref{thm:GTIsufficient_CC}. Here, it should be understood that the
(formal) expression for $t_{\alpha}$ might only be well-defined once
we impose the CC-con\-di\-tion.  Stating the necessary condition in
Theorem~\ref{thm:GTInec_little-l2} in each special case from these
formulas is straightforward and is left to the reader.

\subsection{Gabor systems}
\label{sec:gabor-systems}
Given a countable index set $J$, let
$\{ g_{j} \}_{j \in J} \subset L^2 (G)$. Let $\Gamma \subseteq G$ be a
closed, co-compact subgroup and let $\Lambda \subseteq \widehat{G}$ be
such that equipping it with a $\sigma$-algebra $\Sigma_{\Lambda}$ and
a measure $\mu_{\Lambda}$ gives a measure space
$(\Lambda, \Sigma_{\Lambda}, \mu_{\Lambda})$ satisfying the standard
hypotheses.  The (semi) co-compact Gabor system associated with the
pair $(\Gamma, \Lambda)$ is the collection of functions
\[ 
\{M_{\lambda} T_{\gamma} g_{j} \}_{\lambda \in \Lambda, \gamma \in
  \Gamma, j \in J} = \{\lambda(\cdot) g_{j} (\cdot - \gamma) \}_{
  \lambda \in \Lambda, \gamma \in \Gamma, j \in J}, 
\]
where $M_{\lambda} f (x) = \lambda (x) f (x)$ denotes the modulation
operator on $L^2 (G)$.  The Gabor system
$\{M_{\lambda} T_{\gamma} g_{j} \}_{\lambda \in \Lambda, \gamma \in
  \Gamma, j \in J}$
cannot be expressed as a generalized trans\-la\-tion-in\-vari\-ant
system.  However, since
$\abs{\innerprod{f}{M_{\lambda} T_{\gamma} g_{j}}} =
\abs{\innerprod{f}{T_{\gamma}M_{\lambda} g_{j}}} $
for all $f \in L^2(G)$, the Gabor system
$\{M_{\lambda} T_{\gamma} g_{j} \}_{\lambda \in \Lambda, \gamma \in
  \Gamma, j \in J}$
forms a Bessel system or a frame if, and only if, the corresponding
(generalized) trans\-la\-tion-in\-vari\-ant system
$\{T_{\gamma} M_{\lambda} g_j\}_{\lambda \in \Lambda, \gamma \in
  \Gamma, j \in J}$
with $g_{j,\lambda} = M_{\lambda} g_j$ forms a Bessel system or a
frame.  The auto-correlation functions $t_{\alpha}$ associated with
$\{T_{\gamma} M_{\lambda} g_{j} \}_{\lambda \in \Lambda, \gamma \in
  \Gamma, j \in J}$ can (formally) be written as
\[ t_{\alpha} (\omega) = \sum_{j \in J} \int_{\Lambda} \hat{g}_{j}
(\omega - \lambda) \overline{\hat{g}_{j} (\omega - \lambda - \alpha)}
\; d\mu_{\Lambda} (\lambda) \]
for $\alpha \in \Gamma^{\perp}$. Any trans\-la\-tion-in\-vari\-ant
system satisfying the CC-con\-di\-tion also satisfies the
$\alpha$-LIC. Thus an application of
Theorem~\ref{thm:GTIsufficient_CC} gives the frame bound estimates
\eqref{eq:B_cc-t-alpha} and \eqref{eq:A_cc-t-alpha}, where the $0$th
auto-correlation function $t_0$ is given by
\begin{equation}
t_{0} (\omega) = \sum_{j \in J} \int_{\Lambda} |\hat{g}_{j}
(\omega - \lambda) |^2 \; d\mu_{\Lambda} (\lambda) \label{eq:B0-gabor-freq}
\end{equation}
and the remainder function $R :  \widehat{G} \to [0,\infty]$ by
\begin{equation}
 R(\omega) = \sum_{\alpha \in \Gamma^{\perp} \setminus \{0\}} \bigg|
 \sum_{j \in J} \int_{\Lambda}  \hat{g}_{j} (\omega - \lambda)
 \overline{\hat{g}_{j} (\omega - \lambda - \alpha)} \; d\mu_{\Lambda} (\lambda) \bigg|.
\label{eq:A0-gabor-freq}
\end{equation}
The frame bound estimates associated with \eqref{eq:B0-gabor-freq} and
\eqref{eq:A0-gabor-freq} allow for phase cancellations over the
modulation parameter $\lambda \in \Lambda$. Moreover, if $\Lambda$ is
a closed subgroup, we only need to take the essential supremum and
infimum in \eqref{eq:B_cc-t-alpha} and \eqref{eq:A_cc-t-alpha},
respectively, over a fundamental domain of $\Lambda$ in $\ghat$.  For
singly generated Gabor frames in $L^2 (\mathbb{R}^d)$ associated with
a pair of full-rank lattices $(\Lambda, \Gamma)$, the frame bound
estimates \eqref{eq:B_cc-t-alpha} and \eqref{eq:A_cc-t-alpha} using
\eqref{eq:B0-gabor-freq} and \eqref{eq:A0-gabor-freq} recover
precisely the frame bound estimates by Ron and Shen~\cite{MR1460623,
  MR1350650}.

The sufficient conditions for Gabor frames are often formulated in the
time domain. To do this, we switch the role of $\Gamma$ and $\Lambda$
and consider the Gabor system
$\{M_{\lambda} T_{\gamma} g_{j} \}_{\lambda \in \Lambda, \gamma \in
  \Gamma, j \in J} = \{T_{\lambda} \mathcal{F}^{-1} T_{\gamma} g_{j}
\}_{\lambda \in \Lambda, \gamma \in \Gamma, j \in J}$,
where $\mathcal{F}^{-1}$ denotes the inverse Fourier transform. In
this way, one obtains auto-correlation functions
$s_{\alpha} : G \to \mathbb{C}, \alpha \in \Lambda^{\perp}$, given by
\[ 
s_{\alpha} (x) := \sum_{j \in J} \int_{\Gamma} \overline{g_{j} (x -
  \gamma - \alpha)} g_{j} (x - \gamma) \; d\mu_{\Gamma} (\gamma), 
\]
provided the series converges. 
Hence, if
\[
 B_1 := \esssup_{x \in G} \sum_{\alpha \in
  \Lambda^\perp} \abs{s_{\alpha}(x)} < \infty
\]
and
\[
 A_1 := \essinf_{x \in G} \left( s_0(x) - \sum_{\alpha \in
  \Gamma^\perp\setminus \{0\}} \abs{s_{\alpha}(x)} \right)>0,
\]
then
$\{M_{\lambda} T_{\gamma} g_{j} \}_{\lambda \in \Lambda, \gamma \in
  \Gamma, j \in J}$
is a frame for $L^2(\R^d)$ with bounds $A_1$ and $B_1$. For singly
generated Gabor frames in $L^2 (\mathbb{R}^d)$ associated with a pair
of full-rank lattices $(\Lambda, \Gamma)$, these estimates recover
precisely \cite[Proposition 6.5.5]{MR1843717}.

\subsection{Wavelet systems} 
\label{sec:wavelet-systems}

Let $\aut(G)$ denote the collection of all bi-continuous group
homomorphisms on $G$. For an automorphism $a \in \aut(G)$, let $|a|$
denote its modulus, i.e., the unique positive constant such that
\begin{align*} 
 \int_G f (a (x)) \; d\mu_G (x) = |a| \int_G f (x) \; d\mu_G (x) 
 \end{align*}
for all $f \in L^1 (G)$. Denote by $D_{a} f(x) := |a|^{1/2} f(a(x))$ the unitary dilation operator on $L^2 (G)$.

Let $J$ and $L$ be countable index sets, let
$\{\psi_\ell\}_{\ell \in L} \subset L^2 (G)$, let
$\mathcal{A} := \{a_j \}_{j \in J} \subset \aut(G)$ and let
$\Gamma \subseteq G$ be a closed, co-compact subgroup. The wavelet
system in $L^2 (G)$, associated with the pair $(\mathcal{A}, \Gamma)$,
is the collection of functions
\begin{equation} \label{eq:wavelet-dilation-group}
 \{D_{a} T_{\gamma} \psi_{\ell} \}_{a \in \mathcal{A}, \gamma \in \Gamma,\ell
   \in L} = \{ | a_j |^{1/2} \psi_{\ell} (a_j (\cdot ) - \gamma)
 \}_{j \in J, \gamma \in \Gamma, \ell \in L}. 
\end{equation}
By considering the commutation relation
$D_{a} T_{\gamma} = T_{a^{-1} (\gamma)} D_{a}$ for $a \in \mathcal{A}$
and $\gamma \in \Gamma$, the wavelet system
\eqref{eq:wavelet-dilation-group} can be written as the generalized
trans\-la\-tion-in\-vari\-ant system
$\cup_{j \in J} \{T_{\gamma} g_{j, p} \}_{\gamma \in \Gamma_j}$ with
$\Gamma_j = a_j^{-1}( \Gamma)$ and $g_{j,p} = D_{a_j} \psi_{\ell}$ for
$j \in J$ and $p = \ell \in P$ with $P = L$ equipped with the counting
measure.

The \emph{adjoint} of an automorphism $a \in \aut(G)$ is the
automorphism $\hat{a} : \widehat{G} \to \widehat{G}$ defined by
$\hat{a} (\omega) = \omega \circ a$ for $\omega \in
\widehat{G}$.
Using this notion, the annihilators $\Gamma_j^{\perp}$ of $\Gamma_j$
for $j \in J$ can be written as
$\Gamma_j^{\perp} = (a_j^{-1} (\Gamma))^{\perp} = \hat{a}_j
(\Gamma^{\perp})$,
cf. \cite[Proposition 6.5]{BowRos2014}. For
$\alpha \in \bigcup_{j \in J} \hat{a}_j (\Gamma^{\perp})$, the
auto-correlation function $t_{\alpha} : \widehat{G} \to \mathbb{C}$
can be formally written as
\begin{align*}
  t_{\alpha} (\omega) &= \sum_{\ell \in L} \sum_{j \in \kappa (\alpha)} \frac{1}{\covol(\Gamma_j)} \widehat{(D_{a_j} \psi_{\ell})} (\omega) \overline{\widehat{(D_{a_j} \psi_{\ell})}(\omega + \alpha)} \\
                      &= \sum_{\ell \in L} \sum_{j \in \kappa (\alpha)} \frac{|a_j| }{\covol (\Gamma_j)} \hat{\psi}_{\ell} (\hat{a}_j^{-1} (\omega)) \overline{\hat{\psi}_{\ell} (\hat{a}_j^{-1} (\omega + \alpha))}, 
\end{align*}
where
$\kappa (\alpha) := \{ j \in J \; | \; \alpha \in \hat{a}_j
(\Gamma^{\perp}) \}$.
Observe that $\kappa (0) = J$. Therefore, for wavelet systems
satisfying the 1-UCP, an application of
Theorem~\ref{thm:GTIsufficient_CC} yields the frame bound estimates as
in \eqref{eq:B_cc-t-alpha} and \eqref{eq:A_cc-t-alpha}, where
\begin{align} \label{eq:Calderon-wavelet}
t_0 (\omega) =  \sum_{\ell \in L} \sum_{j \in \kappa (\alpha)} \frac{|a_j|}{\covol (\Gamma_j)}  | \hat{\psi}_{\ell} (\hat{a}_j^{-1} (\omega)) |^2
\end{align}
 is the Calder\'on sum, and the remainder function $R : \widehat{G} \to [0,\infty]$ takes the form
\begin{equation} \label{eq:R-wavelet}
R(\omega)=\sum_{\alpha \in \bigcup_{j \in J} \hat{a}_j (
  \Gamma^{\perp} ) \setminus\{0\}} \biggl\vert \sum_{\ell \in L} \sum_{j \in \kappa (\alpha)}\frac{|a_j| }{\covol (\Gamma_j)}  \hat{\psi}_{\ell} (\hat{a}_j^{-1} (\omega)) \overline{\hat{\psi}_{\ell} (\hat{a}_j^{-1} (\omega + \alpha)) } \bigg\vert . 
\end{equation}
Thus for all generators and for all scales in $\kappa(\alpha)$, we
have the possibility of cancellations in the estimates for each
$\alpha \in \bigcup_{j \in J} \hat{a}_j (\Gamma^{\perp})
\setminus\{0\}$.
This possibility of cancellations is in contrast to known sufficient
conditions and frame bound estimates for wavelet systems based on the
absolute CC-con\-di\-tion. These latter sufficient conditions use the
remainder function $\tilde{R} : \widehat{G} \to [0,\infty]$ given by
\begin{align} 
\label{eq:remainder_absCC} 
  \tilde{R} (\omega) = \sum_{j
  \in J} \frac{|a_j|}{\covol(\Gamma_j)} \sum_{\alpha \in \hat{a}_j(
  \Gamma^{\perp}) \setminus \{0\}} \sum_{\ell \in L} \absBig{
  \hat{\psi}_{\ell} (\hat{a}^{-1}_j \omega)
  \overline{\hat{\psi}_{\ell} (\hat{a}_j^{-1} (\omega + \alpha))} },
\end{align}
in which only the modulus of the generating functions are
considered. To wrap up the discussion, we state the following result.

\begin{theorem} \label{thm:wavelet} Given countable index sets $J, L$,
  let $\{ \psi_{\ell} \}_{\ell \in L} \subset L^2 (G)$, let
  $\{ a_j \}_{j \in J} \subset \aut(G)$ and let $\Gamma \subseteq G$
  be a closed, co-compact subgroup. Suppose the system
  $\{D_{a_j} T_{\gamma} \psi_{\ell} \}_{j \in J, \gamma \in
    \Gamma,\ell \in L}$ satisfies the 1-UCP and satisfies
\begin{align}
\label{eq:B_cc-t-alpha-wavelet} 
b_1 := \esssup_{\omega \in \widehat{G}}  ( t_0 (\omega) + R(\omega) ) < \infty 
\end{align}
and 
\begin{align}
\label{eq:A_cc-t-alpha-wavelet}
a_1 := \essinf_{\omega \in \widehat{G}} ( t_0 (\omega) - R(\omega) ) > 0,
\end{align}
where $t_{0}$ and $R$ are given in \eqref{eq:Calderon-wavelet} and
\eqref{eq:R-wavelet}, respectively. Then
$\{D_{a_j} T_{\gamma} \psi_{\ell} \}_{j \in J, \gamma \in \Gamma,\ell
  \in L}$ forms a frame for $L^2 (G)$ with bounds $a_1$ and $b_1$.
\end{theorem}

The wavelet system in \eqref{eq:wavelet-dilation-group} is defined
with respect to an arbitrary family of automorphisms
$\mathcal{A} \subseteq \aut(G)$. For such general systems, the LIC,
and hence $\alpha$-LIC and $1$-UCP, are not necessarily satisfied
whenever the system satisfies the CC-con\-di\-tion. However, under
additional assumptions on the family $\mathcal{A} \subseteq \aut(G)$,
simple sufficient conditions and characterizations for the LIC are
known. For example, for a family $\{a_j\}_{j \in J} \subset \aut(G)$
for which the adjoints $\{\hat{a}_j \}_{j \in J}$ are \emph{expanding}
in the sense of \cite[Definition 18]{barbieri2017calderon}, the LIC is
automatically satisfied for any system satisfying the
CC-con\-di\-tion. In particular, for a cyclic group
$\mathcal{A} = \langle a \rangle$ generated by $a \in \aut(G)$,
several simple sufficient conditions for the LIC are known. For
example, if the underlying group $G$ possesses a compact open
subgroup, the dilation group can be assumed to be expanding in the
sense of \cite{benedetto2004wavelet, benedetto2004examples}.  For
systems on such groups, the LIC possesses a simple characterization
\cite{MR2283810}.  For a general locally compact abelian group $G$, it
is shown in \cite{MR2283810} that the LIC for wavelet systems
associated to $\mathcal{A} = \langle a \rangle$ is equivalent to
locally integrability of the Calder\'on sum $t_0$, provided that the
adjoint automorphisms are expansive in the sense of \cite[Proposition
4.9]{MR2283810}. See also~\cite[Proposition 2.7]{MR2746669} for the
same result on $G = \mathbb{R}^d$. In this latter setting, the
characterization of the LIC holds in fact for any wavelet system
satisfying the so-called \emph{lattice counting estimate}. In
\cite{BownikNonexpanding2015}, Bownik and the first named author show
that the lattice counting estimate holds for all dilations
$A \in \mathrm{GL}_d (\mathbb{R})$ with $\abs{\det{A}}\neq 1$ and for
almost every translation lattice $\Gamma$ with respect an invariant
probability measure on the set of lattices.  As a consequence,
Theorem~\ref{thm:GTInec_little-l2} and \ref{thm:GTIsufficient_CC} are
applicable for almost all wavelet systems in $L^2(\R^d)$ in the
probabilistic sense of \cite{BownikNonexpanding2015}.

The remainder of this subsection is devoted to two examples for which
phase cancellations in \eqref{eq:R-wavelet} can occur and for which
such cancellations cannot be expected. Both examples take place in
$L^2 (\mathbb{R}^d)$.  In this setting, any automorphism is given by
$x \mapsto Ax$ for some $A \in \mathrm{GL}_d (\mathbb{R})$.  For such
an automorphism, the modulus reads $|\det A|$ and the adjoint is
$A^T$. A discrete, co-compact subgroup $\Gamma \subseteq \mathbb{R}^d$
is a full-rank lattice in $\mathbb{R}^d$, i.e.,
$\Gamma = C \mathbb{Z}^d$ for some $C \in \mathrm{GL}_d
(\mathbb{R})$.
The annihilator $\Gamma^{\perp}$ of a full-rank lattice
$\Gamma \subseteq \mathbb{R}^d$ can be identified with the dual
lattice $\Gamma^{\ast} = C^{\sharp} \mathbb{Z}^d$, where
$C^{\sharp} := (C^T)^{-1}$.
 
\begin{example} \label{ex:transcedental_dilation} Let
  $A \in \mathrm{GL}_d (\mathbb{R})$, let $B := A^T$ and let
  $\Gamma = C \mathbb{Z}^d$ be a full-rank lattice in $\mathbb{R}^d$
  satisfying $\Gamma^\ast \cap B^j \Gamma^\ast = \{0\}$ for all
  $j \in \Z\setminus\{0\}$. Examples of such pairs $(A,\Gamma)$ are
  $B = \beta I$ with $I$ denoting the identity matrix, $\Gamma =\Z^d$,
  and $\beta \in \mathbb{R}$ being such that
  $\beta^j \notin \mathbb{Q}$ for all
  $j \in \mathbb{Z} \setminus \{0\}$. Now, since $B^j \Gamma^\ast$,
  $j \in \Z$, are disjoint outside the origin, it follows that the set
  $\kappa(\alpha)$ is a singleton for each
  $\alpha \in \bigcup_{j \in \mathbb{Z}} B^j
  \Gamma^\ast\setminus\{0\}$.
  Therefore, the remainder function $R : \mathbb{R}^d \to [0,\infty]$
  takes the form
\begin{align*}
  R(\omega)&=\frac{1}{\abs{\det{C}}}\sum_{\alpha \in \bigcup_{j \in \mathbb{Z}} B^j
             \Gamma^\ast\setminus\{0\}} \abs{\sum_{\ell \in L} \hat{\psi}_{\ell} (B^{-j} \omega)
             \overline{\hat{\psi}_{\ell} (B^{-j}(\omega + \alpha))}}  \\
           &= \frac{1}{\abs{\det{C}}}\sum_{j \in \Z} \sum_{k \in \Gamma^\ast\setminus\{0\}} \abs{\sum_{\ell \in L} \hat{\psi}_{\ell} (B^{-j} \omega)
             \overline{\hat{\psi}_{\ell} (B^{-j}\omega + k)}}.
\end{align*}
Consequently, phase cancellation between scales cannot occur in the
estimates in Theorem~ \ref{thm:wavelet}.  This observation fits
precisely with a result by
Laugesen~\cite{MR2280189}. In~\cite{MR2280189}, it is proved that for
wavelet systems in $L^2(\R)$ with transcendental dilations $a > 0$ and
integer translates, which in particular implies that
$\bigcap_{j \in \mathbb{Z}} a^j \Z = \{0\}$, no cancellations between
scales can happen for \emph{any kind} of frame bound estimate based on
$w_f(x)$. Note that despite the fact that no phase cancellations can
happen, the estimate is still optimal for tight frames. This
phenomenon is due to the fact that the characterizing equations for
tight wavelet systems with expansive dilation $A$ satisfying
$\bigcap_{j \in \mathbb{Z}} B^j \Gamma^\ast = \{0\}$ are very
restrictive on properties of $\psi_\ell$. For example, Riesz bases
possessing this property have to be combined MSF
wavelets~\cite{MR1891729,MR1924878, chui2000orthonormal}.
\end{example}

In the previous example it was assumed that the lattices
$B^j \Gamma^\ast$, $j \in \Z$, are disjoint outside the origin. The
next example assumes that the involved lattices are nested.

\begin{example} \label{ex:integer_dilation} Let
  $A \in \mathrm{GL}_d (\mathbb{R})$, let $B := A^T$ and let
  $\Gamma = C \mathbb{Z}^d$ be a full-rank lattice in $\mathbb{R}^d$
  satisfying $B \Gamma^\ast \subset \Gamma^\ast$. In case
  $\Gamma=\Z^d$, this assumption is equivalent with $A$ being
  integer-valued.  The union
  $\bigcup_{j \in \mathbb{Z}} B^j \Gamma^\ast \setminus\{0\} $ can be
  re-written as the \emph{disjoint} union
  $ \bigcup_{m \in \mathbb{Z}} B^m (\Gamma^\ast \setminus
  B\Gamma^\ast)$.
  For $\alpha=B^m q$, where $m\in \Z$ and
  $q\in \Gamma^\ast \setminus B\Gamma^\ast$, we have that
  $ \kappa(\alpha) =\setprop{j \in \Z}{j \le m}. $ Therefore, the
  remainder function $R : \mathbb{R}^d \to [0,\infty]$ takes the form
\begin{align}
R(\omega)
&= \frac{1}{\abs{\det{C}}}\sum_{m \in \Z} \sum_{q \in \Gamma^\ast\setminus B\Gamma^\ast}
  \abs{\sum_{j=-\infty}^m \sum_{\ell \in L} \hat{\psi}_{\ell} (B^{-j} \omega)
  \overline{\hat{\psi}_{\ell} (B^{-j}(\omega + B^m q))}}\nonumber  \\ 
&= \frac{1}{\abs{\det{C}}}\sum_{m \in \Z} \sum_{q \in \Gamma^\ast\setminus B\Gamma^\ast}
  \abs{\sum_{n=0}^\infty \sum_{\ell \in L} \hat{\psi}_{\ell} (B^{n+m} \omega)
  \overline{\hat{\psi}_{\ell} (B^{n}(B^m \omega + q))}} .
\label{eq:R-nested}
\end{align}
Since the functions $t_0$ and $R$ are $B$-dilation periodic, i.e.,
$t_0(B\omega)=t_0(\omega)$ and $R(B\omega)=R(\omega)$ for
a.e. $\omega\in \R^d$, the estimates \eqref{eq:B_cc-t-alpha-wavelet}
and \eqref{eq:A_cc-t-alpha-wavelet} read
 \begin{align}  \label{eq:B_cc-t-alpha_wavelet2}
 b_1 &= \esssup_{\omega \in B(\Omega)\setminus \Omega} \left(\frac{1}{\abs{\det{C}}}\sum_{\ell \in L}\sum_{j \in \Z} \abs{\hat{\psi}_{\ell} (B^j \omega)}^2
 +
   R(\omega)\right) 
\end{align}
and
\begin{align}
 \label{eq:A_cc-t-alpha_wavelet2}
 a_1 &= \essinf_{\omega \in B(\Omega)\setminus \Omega} \left(\frac{1}{\abs{\det{C}}}\sum_{\ell \in L}\sum_{j \in \Z} \abs{\hat{\psi}_{\ell} (B^j \omega)}^2 - R(\omega)\right),
 \end{align}
 where $\Omega:=B(0,1)$ is the unit ball in $\R^d$, and
 $R : \mathbb{R}^n \to [0,\infty]$ is given as in \eqref{eq:R-nested}.
 For univariate wavelets with $A=B=2$ and $\Gamma=c\Z$, $c>0$, these
 estimates coincide\footnote{The frame bound estimates
   \eqref{eq:B_cc-t-alpha_wavelet2} and
   \eqref{eq:A_cc-t-alpha_wavelet2} are slightly improved versions of
   the estimates that occur in~\cite[Theorem 2.9]{Dau90}. The
   improvement boils in essence down to a change of variable and
   taking suprema and infima differently than in the original proof.}
 with Tchamitchian's estimates as communicated by
 Daubechies~\cite{Dau90, Dau92}

 To show that the frame bound estimates from Theorem~\ref{thm:wavelet}
 improve the sufficient condition based on the remainder function
 \eqref{eq:remainder_absCC}, note that \eqref{eq:remainder_absCC} in
 the special case considered in this example simply reads
\begin{equation*}
\tilde{R}(\omega)= \frac{1}{\abs{\det{C}}}\sum_{j \in \Z} \sum_{\alpha \in \Gamma^\ast\setminus \{0\}}
  \sum_{\ell \in L} \abs{\hat{\psi}_{\ell} (B^{j} \omega)
  \overline{\hat{\psi}_{\ell} (B^{j}\omega + \alpha)}}. 
\end{equation*}
Now to see that $R(\omega) \le \tilde{R}(\omega)$ for a.e.
$\omega\in \R^n$, one simply uses the triangle inequality and notes
that there is a bijection between the indices
$(m,n,q) \in (\Z,\N,\Gamma^\ast\setminus B\Gamma^\ast)$ and the
indices $(j,\alpha)\in \Z\times \Gamma^\ast\setminus\{0\}$ given by
\[
(m,n,q) \mapsto (j,\alpha), \quad \text{where } \alpha=B^nq \text{ and } j=n+m. 
\]
\end{example}

For $\Gamma=\Z^d$, the above two examples show the two extremes on the
possible phase cancellations of Theorem~\ref{thm:GTIsufficient_CC}
that happen for integer dilations and certain irrational dilations.
For a rational dilation matrix $A\in \mathrm{GL}_d (\Q)$, frame bound
estimates with phase cancellations in \eqref{eq:R-wavelet} over
infinitely many scales are clearly also possible.  In fact, Laugesen
remarked already in~\cite{MR2280189} that this would be possible for
rational dilation in dimension one, such dilations necessarily being
expansive. Recall that the analysis in the present paper does not
require that the dilation is expansive, only that the $1$-UCP is
satisfied.

\subsection{Composite wavelets and shearlet systems} 
\label{sec:shearlet-systems}

Consider the Cartesian product $I \times J$ for two countable index
sets $I$ and $J$. Let $ A_i, B_j \in \mathrm{GL}_d (\mathbb{R})$ for
$i \in I$ and $j \in J$. Let $\Gamma = C \mathbb{Z}^d$ be a full-rank
lattice in $\mathbb{R}^d$. The wavelet system associated with the pair
$(\{A_i B_j\}_{(i,j) \in I \times J}, \Gamma)$ is a collection of
functions of the form
  \[ \{ D_{A_i B_j} T_{\gamma} \psi_{\ell} \}_{i \in I, j \in J,
    \gamma \in \Gamma, \ell \in L} 
\]
  and forms a so-called \emph{wavelet system with composite dilations}
  in $L^2 (\mathbb{R}^d)$, see e.g., \cite{MR2207836}. One usually
  assumes that one of the two family of matrices, say
  $\set{A_i}_{i \in I}$, is volume preserving. We will assume that
  $A_i^T$, $i \in I$, acts invariant on $\Gamma^\ast$, that is,
  $A_i^T \Gamma^\ast = \Gamma^\ast$, e.g., in case $\Gamma=\Z^d$, this
  assumption reads $A_i \in \mathrm{SL}_d (\mathbb{Z})$.  Therefore,
  $\Gamma_{(i,j)}^{\perp} = B_j^T A_i^T \Gamma^* =B_j^T \Gamma^*$ for
  $(i,j) \in I \times J$. Thus, for composite wavelet systems
  satisfying the 1-UCP, an application of Theorem
  \ref{thm:wavelet} yields the frame bound estimates
  \eqref{eq:B_cc-t-alpha-wavelet} and \eqref{eq:A_cc-t-alpha-wavelet}, where
  \[ t_0 (\omega) = \frac{1}{|\det C|} \sum_{i \in I} \sum_{j \in J}
  \sum_{\ell \in L} \Bigl| \hat{\psi}_{\ell} ( A^{\sharp}_{i}
  B_j^{\sharp} \omega) \Bigr|^2 \] and
  \[ R(\omega) = \frac{1}{|\det C|} \sum_{\alpha \in \bigcup_{j \in J}
    B_j^T \Gamma^* \setminus \{0\}} \biggl\vert \sum_{\ell \in L} \sum_{i
    \in I} \sum_{j \in \kappa(\alpha)}\hat{\psi}_{\ell} (
  A^{\sharp}_{i} B_j^{\sharp} \omega) \overline{\hat{\psi}_{\ell} (
    A^{\sharp}_{i} B_j^{\sharp} (\omega + \alpha))} \biggr\vert \]
  with
  $\kappa (\alpha) := \{ j \in J \; | \; \alpha \in B_j^T \Gamma^*
  \setminus \{0\} \}$.

The classical shearlet system is a special case of wavelets with
composite dilations. For simplicity we restrict our attention to
$L^2 (\mathbb{R}^2)$, but we refer to \cite[section 3.4]{MR2207836} for a
discussion of shearlet systems in $L^2 (\mathbb{R}^d)$. One defines
\[ A=\begin{pmatrix}
4 & 0 \\
0 & 2
\end{pmatrix}
\quad
\text{and}
\quad
S
= 
\begin{pmatrix}
1 & 1 \\
0 & 1
\end{pmatrix},
\]
and considers the wavelet system associated with the pair
$(\{S^k A ^j\}_{j,k\in\Z},\Gamma)$, where $\Gamma=C\Z^2$ for some
$C \in \mathrm{GL}_d(\R)$. For the \emph{classical shearlet system} of
the form
$\{ D_{S^k A^j } T_{\gamma} \psi_\ell \}_{j,k \in \Z, \gamma \in
  \Gamma, \ell \in L}$
we find as above that the corresponding functions
$t_0 : \mathbb{R}^2 \to \mathbb{C}$ and $R : \R^2 \to [0,\infty]$ are
formally given as
  \begin{align} \label{eq:shearlet_calderon}
   t_0 (\omega) =  \frac{1}{\abs{\det{C}}}\sum_{j \in \Z} \sum_{k \in \Z}  \sum_{\ell \in L} 
   \Bigl| \hat{\psi}_\ell( (S^{\sharp})^k  A^{-j}
  \omega) \Bigr|^2 
  \end{align} 
  and
  \begin{align} \label{eq:shearlet_R}
   R(\omega) = \frac{1}{|\det C|} \sum_{m \in \Z} \sum_{q \in \Gamma^\ast\setminus A\Gamma^\ast}
  \abs{\sum_{n=0}^\infty  \sum_{k \in \Z}  \sum_{\ell \in
      L}\hat{\psi}_\ell ((S^{\sharp})^k A^{n+m}  \omega)
  \overline{\hat{\psi}_\ell ((S^{\sharp})^k A^{n}(A^m \omega + q))}} .
 \end{align}
 Since any shearlet system that satisfies the CC-con\-di\-tion
 satisfies the $\alpha$-LIC, an application of
 Theorem~\ref{thm:GTIsufficient_CC} yields the following result.

\begin{theorem} \label{thm-shearlet}
Let $L$ be a countable index set, let $\{\psi_{\ell} \}_{\ell \in L} \subset L^2 (\R^2)$ and let $\Gamma \subset \mathbb{R}^2$ be a full-rank lattice. Suppose the shearlet system $\{ D_{S^k
  A^j } T_{\gamma} \psi_\ell \}_{j,k \in \Z, \gamma \in \Gamma, \ell
  \in L}$ satisfies 
  \[ b_1 := \esssup_{\omega \in \R^2} (t_0(\omega) + R(\omega)) < \infty \]
  and
  \[ a_1 := \essinf_{\omega \in \R^2} (t_0 (\omega) - R(\omega)) > 0, \]
  where $t_0$ and $R$ are given in \eqref{eq:shearlet_calderon} and \eqref{eq:shearlet_R}, respectively.
 Then the shearlet system $\{ D_{S^k
  A^j } T_{\gamma} \psi_\ell \}_{j,k \in \Z, \gamma \in \Gamma, \ell
  \in L}$ forms a frame for $L^2 (\R^2)$ with bounds $a_1$ and $b_1$.
\end{theorem}
The estimates in Theorem \ref{thm-shearlet} should be compared with
previously used sufficient conditions for shearlet systems that are
based on the absolute CC-con\-di\-tion and that do not allow for phase
cancellations~\cite{kutyniok2007construction}.

The rest of this subsection is devoted to cone-adapted shearlet
systems.  Such shearlets play a more important role in applications
than the classical shearlets as they treat directions in an almost
uniform manner. The cone-adapted shearlet system is a finite union of
shift-invariant systems and wavelet systems with composite
dilations. To introduce these systems, we define $A_1=A$, $S_1=S$,
\[
A_2= \begin{pmatrix}
2 & 0 \\
0 & 4
\end{pmatrix}
\quad
\text{and}
\quad
S_2
= 
\begin{pmatrix}
1 & 0 \\
1 & 1
\end{pmatrix}.
\]
For generators $\phi, \psi_i \in L^2(\R^2)$, $i=1,2$, and full-rank lattices 
$\Gamma_i=C_i\Z^2$, $i = 0,1,2$, the \emph{cone-adapted
shearlet system} is given as:
\[ 
\{T_{\gamma} \phi \}_{\gamma \in \Gamma_0} \; \cup \; \{ D_{S_i^k
  A_i^j } T_{\gamma} \psi_i \}_{j \in \mathbb{N}_0, k \in \{-K_j,\dots,K_j\},
  \gamma \in \Gamma_i, i\in\{1,2\}},
\]
where $K_j \in \N_0$ for $j \in \N_0$, usually one takes $K_j = 2^j$
or $K_j = 2^j\pm 1$.

For brevity we assume $\Gamma_i=\Gamma=C\Z^2$ for $i=0,1,2$ for some
$C \in \mathrm{GL}_d(\R)$ so that $C^T A_i C^\sharp$ is integer valued
for $i \in \{1,2\}$. 
The auto-correlation functions $t_{\alpha} : \R^2 \to \mathbb{C}$, $\alpha \in \Gamma^*$,  are then 
formally
given as:
\begin{align}
 t_0 (\omega) &= \absbig{\hat\phi(\omega)}^2
  + \sum_{i\in \{1,2\}} \sum_{j=0}^{\infty}
  \sum_{k= - K_j}^{K_j}\absbig{\hat{\psi}_i ((S_i^{\sharp})^k
                A_i^{-j}  \omega)}^2, 
\label{eq:t_0_cone-shearlet} \\
  t_\alpha(\omega) &=  \hat\phi(\omega)
  \overline{\hat\phi(\omega+\alpha)} + \sum_{i\in \{1,2\}} \sum_{j=0}^{m_i}
  \sum_{k= - K_j}^{K_j} \hat{\psi}_i ((S_i^{\sharp})^k A_i^{-j}  \omega)
  \overline{\hat{\psi}_i ((S_i^{\sharp})^k A_i^{-j}( \omega +
    \alpha))}, \label{eq:t_alpha_cone-shearlet}
\end{align}
where $\alpha \in \Gamma^\ast\setminus \{0\}$, for each
$i\in \{1,2\}$, is written as $A_i^{m_i}q_i$ for unique $m_i\ge 0$ and
$q_i \in \Gamma^\ast \setminus A_i \Gamma^\ast$.  From the
auto-correlation functions \eqref{eq:t_alpha_cone-shearlet} we see
that for the cases $\alpha \in C^\sharp \Z^2\setminus 2C^\sharp \Z^2$
and $\alpha \in C^\sharp (4\Z^2+(2,2))$, the least amount of
cancellation is possible. In this case the auto-correlation function
reads
\[
  t_\alpha(\omega) =  \hat\phi(\omega)
  \overline{\hat\phi(\omega+\alpha)} + \sum_{i\in \{1,2\}} \sum_{k= -K_0}^{K_0} \hat{\psi}_i ((S_i^{\sharp})^k  \omega)
  \overline{\hat{\psi}_i ((S_i^{\sharp})^k( \omega + \alpha))},
\]
hence only cancellation within the $0$th scale is possible. On the
other hand, when $\alpha \in 4^p C^\sharp \Z^2$ for some $p \in \N$, then
cancellations can happen within all shears and all scales
$j=0,\dots,p$ for \emph{both} shearlet generators $\psi_1$
and $\psi_2$, that is, $m_1=m_2=p$ in \eqref{eq:t_alpha_cone-shearlet}.  

As local integrability conditions can be ignored for shearlet systems,
we arrive at the following Tchamitchian-type estimate for cone-adapted
shearlet systems. 
\begin{theorem} \label{thm:cone-shearlet}
  Let $\phi, \psi_i \in L^2(\R^2)$, $i=1,2$, and let $\Gamma$ be a
  full-rank lattice in $\R^2$. If 
 \begin{align}  \label{eq:B_cc-t-alpha_cone_shear}
 b_1 &:= \esssup_{\omega \in \R^2} \sum_{\alpha \in
  \Gamma^\ast} \abs{t_{\alpha}(\omega)} < \infty
\end{align}
and
\begin{align}
 \label{eq:A_cc-t-alpha_cone_shear}
 a_1 &:= \essinf_{\omega \in \R^2} \left( t_0(\omega) - \sum_{\alpha \in
  \Gamma^\ast\setminus \{0\}} \abs{t_{\alpha}(\omega)} \right)>0,
 \end{align}
where $t_\alpha$ is given by \eqref{eq:t_0_cone-shearlet} and \eqref{eq:t_alpha_cone-shearlet},
then the cone-adapted  shearlet system 
\[ 
\{T_{\gamma} \phi \}_{\gamma \in \Gamma} \; \cup \; \{ D_{S_i^k
  A_i^j } T_{\gamma} \psi_i \}_{j \in \mathbb{N}_0, k \in \{-K_j,\dots,K_j\},
  \gamma \in \Gamma, i\in\{1,2\}}. 
\]
is a frame for $L^2(\R^2)$ with bounds $a_1$ and $b_1$.
\end{theorem}

The estimates in Theorem~\ref{thm:cone-shearlet} are improvements of
the sufficient conditions for cone-adapted shearlet systems as given
in~\cite{MR2864368}, which are based on the absolute CC-con\-di\-tion
and do not allow for phase cancellations. Here, it should be noted
that the conditions in~\cite{MR2864368} are currently the only known
method for constructing cone-adapted shearlet frames with compactly
supported generators. Moreover, the estimates without phase
cancellation in~\cite{MR2864368} are used to ``optimize'' the choice
of shearlet and translation lattice. It would be beneficial to instead
use the improved estimates \eqref{eq:B_cc-t-alpha_cone_shear} and
\eqref{eq:A_cc-t-alpha_cone_shear} for optimizing the construction of
compactly supported shearlets.

\subsection{Continuous trans\-la\-tion-in\-vari\-ant systems}
\label{sec:continuous-transform}

This section considers ``continuous'' trans\-la\-tion-in\-vari\-ant
systems with translation along the whole group, e.g., $J$ being a
singleton and $\Gamma = G$. Since $G^{\perp} = \{0\}$, there is only
one correlation function $t_0 : \widehat{G} \to \mathbb{C}$, and since
$J$ is a singleton, the $\infty$-UCP trivially holds. Therefore, by
combining the necessary condition $A \le t_0 \le B$ from
\cite{bandwidth} and Theorem~\ref{thm:GTIsufficient_CC}, we
immediately recover the following characterization of the frame
property \cite{iverson2015subspaces,MR344891,1751-8121-48-39-395201}.

\begin{corollary} \label{cor:TI_characterization}
Let $0 < A \leq B < \infty$ and let $ \{T_{\gamma} g_{ p} \}_{\gamma \in
  G, p \in P}$ be a generalized trans\-la\-tion-in\-vari\-ant system
satisfying the standing hypotheses \eqref{item:hyp1}--\eqref{item:hyp3}. 
The system $\{T_{\gamma} g_{p} \}_{\gamma \in
  G, p \in P}$ forms a frame for $L^2 (G)$ with frame bounds $A$ and
$B$ if, and only if,

\[
  A \le  \int_{P} \abs{\hat{g}_{p}(\omega)}^2 d
  \mu_{P}(p) \le B 
\]
for $\mu_{\widehat{G}}$-a.e. $\omega \in \widehat{G}$.
\end{corollary}

For continuous trans\-la\-tion-in\-vari\-ant systems, being a frame is
equivalent to the transform
$C : L^2(G) \to L^2(P \times G), \; f \mapsto
\{\innerprod{f}{T_{\gamma} g_{p}} \}_{\gamma \in G, p \in P}$
being an injective, bounded linear operator with closed
range. Classical examples of such transforms are the continuous
wavelet transform and the windowed Fourier transform. However, the
continuous bendlet transform or, more generally, the $\ell$-th order
$\alpha$-shearlet transform, recently introduced in \cite{LESSIG2017},
are also examples of trans\-la\-tion-in\-vari\-ant transforms. For
these higher-order shearlet transforms the representation-theoretic
approach, utilizing orthogonality relations for irreducible,
square-integrable representations of an associated locally compact
group, is not directly applicable \cite[Section 5]{LESSIG2017}. Since
no characterizations of the frame property of the higher-order
$\alpha$-shearlet transform are known, we outline in the next example
how such a characterization can be obtained from Corollary
\ref{cor:TI_characterization}.

\begin{example}
  Let $G=\R^2$. Define the $\alpha$-scaling operator
  $A_a:\R^2 \to \R^2$ by $A_a (x_1,x_2) = (ax_1,a^\alpha x_2)$ for
  $\alpha \in \itvcc{0}{1}$ and $a>0$, and define the $\ell$-th order
  (non-linear) shearing operator $S_r : \R^2 \to \R^2$ by
  $S_r(x_1,x_2)=(x_1 + \sum_{m=1}^\ell r_m x_2^m, x_2)$ for
  $r=(r_1,\dots,r_{\ell})\in \R^\ell$. The Jacobian determinants of
  $A_a$ and $S_r$ are $a^{1+\alpha}$ and $1$, while the inverses are
  $A_{a^{-1}}$ and $S_{-r}$, respectively.

  Let $P=\R_{> 0} \times \R^\ell$ and set
  $g_p=a^{-(1+\alpha)/2} \psi(A_{a^{-1}}S_{-r}\cdot)$ for some
  $\psi \in L^2(\R^2)$ and $p=(a,r)\in P$. The continuous $\ell$-th
  order $\alpha$-shearlet transform is simply the system
  $\{T_{\gamma} g_{p} \}_{\gamma \in G, p \in P}$, which reads
\[ 
  \set{ a^{-(1+\alpha)/2} \psi(A_{a^{-1}}S_{-r}(\cdot-\gamma))}_{a\in
    \R_{> 0}, r \in \R^\ell,\gamma \in \R^2}.
\]
By Corollary \ref{cor:TI_characterization}, the system forms a frame
with bounds $A$ and $B$ if, and only if,
\[
  A \le \int_0^\infty \int_{\R^\ell} a^{-(1+\alpha)}
  \abs{\psi(A_{a^{-1}}S_{-r}(\cdot))^{\wedge}(\omega)}^2 dr da \le B
  \quad  \text{for a.e. } \omega \in \R^2.
\]
Here, we have not specified the measure $dr da$ on $P$; a canonical
choice is $a^{-\ell-2+\alpha(\ell-1)}$ times the Lebesgue measure on
$\R_{> 0} \times \R^\ell$, but the characterization is valid for any
measure on $P$ satisfying the standing hypotheses
\eqref{item:hyp1}--\eqref{item:hyp3}.

The cone-adapted version is obtained by equipping
$\setprop{(a,r)}{a\in \itvoc{0}{1}, r \in R}$ with a measure $drda$
(satisfying the standing hypotheses), where $R$ is a subset of
$\R^\ell$; a canonical choice being
$R = \itvcc{-1-a^{1-\alpha}}{1+a^{1-\alpha}} \times \R^{\ell-1}$. Let
$Q:\R^2 \to \R^2$ be the permutation defined by
$Q(x_1,x_2)=(x_2,x_1)$, let $\tilde{A}_a=Q \circ A_a \circ Q$, and
$\tilde{S}_r=Q \circ S_r \circ Q$. The cone-adapted continuous
$\ell$-th order $\alpha$-shearlet system generated by
$\phi, \psi, \tilde{\psi}\in L^2(\R^2)$ is given by
\begin{multline*}
  \set{\phi(\cdot - \gamma)}_{\gamma \in \R^2} \cup \set{
    a^{-(1+\alpha)/2} \psi(A_{a^{-1}}S_{-r}(\cdot-\gamma))}_{a\in
    \itvoc{0}{1}, r \in R,\gamma \in \R^2} \\ \cup \set{
    a^{-(1+\alpha)/2}
    \tilde{\psi}(\tilde{A}_{a^{-1}}\tilde{S}_{-r}(\cdot-\gamma))}_{a\in
    \itvoc{0}{1}, r \in R,\gamma \in \R^2},
\end{multline*}
and forms a frame for $L^2(\R^2)$ with bounds $A$ and $B$ if, and only if,
\begin{align*}
  A \le \absbig{\hat{\phi}{(\omega)}}^2 &+ \int_0^1 \int_{R} a^{-(1+\alpha)}     \abs{\psi(A_{a^{-1}}S_{-r}(\cdot))^{\wedge}(\omega)}^2
                                          dr da  \\ & + \int_0^1 \int_{R} a^{-(1+\alpha)}
                                                      \abs{\tilde{\psi}(\tilde{A}_{a^{-1}}\tilde{S}_{-r}(\cdot))^{\wedge}(\omega)}^2
                                                      dr
                                                      da
                                                      \leq
                                                      B                                                                                                        
\end{align*}
for a.e. $\omega \in \mathbb{R}^2$.
\end{example}

\paragraph{Acknowledgments.} The authors would like to thank the
referees for comments improving the presentation of the paper. The
second named author gratefully acknowledges support from the Austrian
Science Fund (FWF): P29462-N35.


\end{document}